\documentclass[11pt]{amsart}
\usepackage[top=1.2in, bottom=1.2in, left=1.5in, right=1.5in]{geometry}
\usepackage{bbm}

\usepackage{amsfonts, amssymb, amsmath, enumerate}

\numberwithin{equation}{section}

\DeclareMathOperator{\E}{\mathbb{E}}

\DeclareMathOperator*{\diag}{diag}
\DeclareMathOperator*{\Span}{span}
\DeclareMathOperator{\Det}{det}

\DeclareMathOperator{\HS}{HS}
\DeclareMathOperator{\sign}{sign}

\DeclareMathOperator{\Row}{Row}
\DeclareMathOperator{\Col}{Col}
\DeclareMathOperator{\Ker}{Ker}
\DeclareMathOperator{\dist}{dist}
\DeclareMathOperator{\rank}{rank}
\DeclareMathOperator{\Sparse}{Sparse}
\DeclareMathOperator{\Comp}{Comp}
\DeclareMathOperator{\Incomp}{Incomp}

\def \N {\mathbb{N}}
\def \P {\mathbb{P}}
\def \R {\mathbb{R}}

\def \Z {\mathbb{Z}}

\def \EE {\mathcal{E}}

\def \LL {\mathcal{L}}
\def \MM {\mathcal{M}}
\def \NN {\mathcal{N}}
\def \VV {\mathcal{V}}

\def \a {\alpha}

\def \e {\varepsilon}
\def \d {\delta}

\newcommand{\etc}{,\ldots,}

\newcommand{\norm}[1]{\left \| #1 \right \|}

\newtheorem{theorem}{Theorem}[section]
\newtheorem{proposition}[theorem]{Proposition}
\newtheorem{corollary}[theorem]{Corollary}
\newtheorem{lemma}[theorem]{Lemma}

\newtheorem{definition}[theorem]{Definition}

\theoremstyle{remark}
\newtheorem{remark}[theorem]{Remark}

\title[A large deviation inequality for the rank of a random matrix]{A large deviation inequality for the rank of a random matrix}

\author{Mark Rudelson}

\address{Department of Mathematics, University of Michigan, 530 Church St., Ann Arbor, MI 48109, U.S.A.}
\email{rudelson@umich.edu}
\thanks{Research supported in part by NSF grant  DMS 2054408 and a fellowship from the Simons Foundation.}

\date{\today}


\begin{document}
	
	\begin{abstract}
		Let $A$ be an $n \times n$ random matrix with independent identically distributed non-constant  subgaussian entries.Then for any $k \le c \sqrt{n}$,
		\[
		\rank(A) \ge n-k
		\]
		with probability at least $1-\exp(-c'kn)$.
	\end{abstract}
	
	\maketitle

\section{Introduction} \label{sec: Introduction}

 Estimating the probability that an $n \times n$ random matrix with independent identically distributed (i.i.d.) entries is singular is a classical problem in probability. 
 The first result in this direction showing that for a matrix with Bernoulli$(1/2)$ entries, this probability is $O(n^{-1/2})$ was proved by Koml\'os \cite{Komlos} in 1967. 
 In a breakthrough paper \cite{KKS}, Kahn, Koml\'os, and Szemer\'edi established the first exponential bound for Bernoulli matrices: 
 \[
   \P(\det(A_n)=0)=(0.998+o(1))^n.
 \]
 The asymptotically optimal exponent has been recently obtained by Tikhomirov \cite{Tikh}:
 \[
  \P(\det(A_n)=0)=\left(\frac{1}{2}+o(1)\right)^n.
 \] 
  The exponential bound for probability of singularity holds in a more general context than Bernoulli random matrices.
  It was proved in \cite{RV invertibility} for matrices with i.i.d. subgaussian entries and extended in \cite{RT} to matrices whose entries have bounded second moment.
  
  A natural extension of the question about the probability of singularity is estimating the probability that a random matrix has a large co-rank.
  More precisely, we are interested in the asymptotic of $\P(\rank(A_n) \le n-k)$, where $k <n$ is a number which can grow with $n$ as $n \to \infty$.
  Such rank means that there are $k$ columns of $A_n$ which are linearly dependent on the other columns.
  Based on the fact that 
  \[
    \P(\rank(A_n) \le n-1) =  \P(A_n \text{ is singular}) \le \exp(-cn),
  \]
  and the independence of the columns of $A_n$, one can predict that the probability that the rank of $A_n$ does not exceed $n-k$ is bounded by $ \big(\exp(-cn)\big)^k= \exp(-cnk)$.
  Proving such a bound amounts to obtaining a super-exponential probability estimate if $k \to \infty$ as $n \to \infty$. This makes a number of key tools in the previously mentioned papers unavailable, because these tools were intended to rule out pathological events of probability $O(\exp(-cn))$ which cannot be considered negligible in this context.
  
  The existing results fell short of this tight bound until recently. 
  Kahn, Komlos, and Szemeredi showed that the probability that a Bernoulli$(1/2)$ matrix has rank smaller than $n-k$ is $O(f(k)^n)$ where $f(k) \to 0$ as $k \to \infty$.
  The intuitive prediction above was recently confirmed by Jain, Sah and Sawney in the case when $k \in \N$ is a fixed number. 
  Building on the ideas of Tikhomirov \cite{Tikh}, they proved an optimal bound for random matrices with independent Bernoulli$(p)$ entries. 
  Namely, for any $p \in (0,1/2], \ \e>0$, and for any $n> n_0(k, p,\e)$
  \[
    \P(\rank(A_n) \le n-k) \le \left(1-p+ \e \right)^{kn}.
  \]
  This completely solves the problem for Bernoulli matrices within the exponential range.
  However, the methods of this paper do not seem to be extendable to the case when $k$ grows together with $n$, i.e., to the super-exponential range of probabilities (see Section \ref{subsec: Outline} for more details).
  
  The main result of this paper confirms this prediction in the super-exponential range for all matrices with  i.i.d. subgaussian entries.
  A random variable $\xi$ is called subgaussian if 
  \[
    \E \exp \left(-\left(\frac{\xi}{K} \right)^2 \right) < \infty
  \]
  for some $K>0$.
  In what follows, we regard $K$ as a constant and allow other constants such as $C,c, c'$, etc. depend on it.
  This is an rich class of random variables including, for instance, all bounded ones.
  
  We prove the following theorem.

 \begin{theorem} \label{thm: rank}
 	Let $k,n \in \N$ be numbers such that $k \le c n^{1/2}$.
 	Let $A$ be an $n \times n$ matrix with i.i.d. non-constant subgaussian  entries.
 	Then
 	\[
 	  \P \left( \rank(A) \le n-k \right) \le \exp (-c'k n).
 	\]
 \end{theorem}


\begin{remark}
	The bound of Theorem \ref{thm: rank} should hold for $k > c n^{1/2}$ as well. The restriction on $k$ in the theorem arises from using $\nu$-almost orthogonal systems throughout its proof, see Definition \ref{def: nu-orthogonal} and Remark \ref{rem: nu-orthogonal}.
\end{remark}
 
 \begin{remark}
 	Combining the technique of this paper with that of Nguyen \cite{Nguyen}, one can also obtain a lower bound for the singular value $s_{n-k}(A_n)$ of the same type as in \cite{Nguyen} but with the additive error term $\exp(-ckn)$ instead of $\exp(-cn)$. 
 	We will not pursue this route in order to keep the presentation relatively simple.
 \end{remark}

  The importance of getting the large deviation bound of Theorem \ref{thm: rank} in the regime when $k$ grows simultaneously with $n$ stems in particular from its application to   \emph{ Quantitative Group Testing} (QGT).
   This computer science problem considers a collection of $n$ items containing $k$ defective ones, where $k < n$ is regarded as a known number. 
   A test consists of selecting a random pool of items choosing each one independently with probability $1/2$ and outputting the number of defective items in the pool.
   The aim of the  QGT is to efficiently determine the defective items after a small number of tests. 
   The question of constructing an efficient algorithm for QGT is still open.
   In \cite{FL}, Feige and Lellouche introduced the following relaxation of the QGT: after $m>k$ tests, one has to produce a subset $S \subset [n]$ of cardinality $m$, containing all defective items. 
   This means that unlike the original QGT, the approach of Feige and Lellouche allows false positives which makes the problem simpler and admits more efficient algorithms. 
   Denote by $A$ the $m \times n$ matrix whose rows are the indicator functions of the tests, and denote by $A|_S$ its submatrix with columns from the set $S \subset [n]$.
   Then $A$ is a random matrix with i.i.d. Bernoulli entries.
   The main result of \cite{FL} asserts that if an algorithm for the relaxed problem succeeds and outputs a set $S \subset [n]$ and $\rank(A|_S) \ge m - O(\log n)$, then one can efficiently determine the set of defective items.
   Checking this criterion for a given algorithm is difficult since the set $S$ is not known in advance.
   However, if we know that 
    \begin{equation} \label{eq: min-rank}
   	\rank(A|_S) \ge m - O(\log n)
   \end{equation}
    for all $m$-element sets $S \subset [n]$ at the same time, this condition would be redundant, and all algorithms for the relaxed problem could be adapted to solve the QGT. 
   In other words, we need to estimate the minimal rank of all $m \times m$ submatrices of an $m \times n$ random matrix.
   We show below that Theorem~\ref{thm: rank} implies that the bound \eqref{eq: min-rank} holds with high probability, and moreover, that this is an optimal estimate  (see Lemma \ref{lem: submatrices}).

\section{Notation and the outline of the proof} \label{sec: Notation and Outline}
 \subsection{Notation} \label{subsec: notation}
 We denote by $[n]$ the set of natural numbers from $1$ to $n$.
 Given a vector $x \in \R^n$, we denote by $\norm{x}_2$ its standard Euclidean norm: $\norm{x}_2= \left(\sum_{j \in [n]} x_j^2 \right)^{1/2}$.
 The unit sphere of $\R^n$ is denoted by $S^{n-1}$.
 
 If $V$ is an $m \times l$ matrix, we denote by $\Row_i(V)$ its $i$-th row and by $\Col_{j}(V)$ its $j$-th column.
 Its singular values will be denoted by
 \[
   s_1(V) \ge s_2(V) \ge \cdots \ge s_m(V) \ge 0.
 \]
 The operator norm of $V$ is defined as 
 \[
   \norm{V}=\max_{x \in S^{l-1}} \norm{V x}_2,
 \]
 and the Hilbert-Schmidt norm as
 \[
 \norm{V}_{\HS}= \left(\sum_{i=1}^m \sum_{j=1}^l v_{i,j}^2 \right)^{1/2}.
 \]
 Note that $\norm{V}=s_1(V)$ and $\norm{V}_{\HS}= \left(\sum_{j=1}^m s_j(V)^2\right)^{1/2}.$
 
 Throughout the paper, the letters $c, \bar{c}, C$ etc. stand for absolute constants whose values may change from line to line.

  \subsection{Outline of the proof} \label{subsec: Outline}
  
   Let $A$ be an $n \times n$ random matrix with i.i.d. entries.
   The fact that this matrix has rank at most $n-k$ means that at least $k$ of its columns are linearly dependent on the rest. 
   Assume that the $k$ last columns are linearly dependent on the other.
   As the results of \cite{RV invertibility} show, for a typical realization of the first $n-k$ columns, the probability that a given column belongs to their linear span is $O(\exp(-cn) )$. 
   Since the last $k$ columns are mutually independent and at the same time independent of the first $n-k$ ones, the probability that all $k$ columns fall into the linear span of the rest is $O \Big( (\exp(-c n ))^k \Big) = O(\exp(-cnk))$ which is the content of our main theorem.
   
   The problem with this argument, however, is in the meaning of the term ``typical''.
   It includes several requirements on the matrix with these $n-k$ rows, including that its norm is $O(\sqrt{n})$ and that its kernel contains no vector with a rigid arithmetic structure. As was shown in  \cite{RV invertibility}, all these requirements hold with probability at least $1-\exp(-cn)$ which is enough to derive that the matrix is invertible with a similar probability.
   In our case, when we aim at bounding probability by $\exp(-ckn)$ with $k$ which can tend to infinity with $n$, the events which have just exponentially small probability cannot be considered negligible any longer. 
   In particular, we are not able to assume that the operator norm of a random matrix is bounded by $O(\sqrt{n})$. 
   This is, however, the easiest of the arising problems as we will be able to use a better concentrated Hilbert-Schmidt norm instead.
   
   The problem of ruling out the arithmetic structure of the kernel turns out to be more delicate. 
   For Bernoulli$(p)$ random matrices with $0<p \le \frac{1}{2}$, Jain, Sah, and Sawney \cite{JSS rank} overcame it by replacing the approach based on the least common denominator used in \cite{RV invertibility} with a further development of the averaging method of Tikhomirov \cite{Tikh}. 
   This allowed them to prove that if $k$ is a constant, then with probability $1-4^{-kn}$, the kernel of the matrix consisting of the first $n-k$ rows either consists of vectors close to sparse (compressible), or does not contain \emph{any} vector with a problematic arithmetic structure, see \cite[Proposition 2.7]{JSS rank} whose proof follows \cite[Proposition 3.7]{JSS part II}. 
   They further derived from this fact that the probability that a random Bernoulli matrix has rank $n-k$ or smaller does not exceed
   \[
    \left(1-p+o(1)\right)^{kn}
   \]
   for any constant $k$.
   However, this approach is no longer feasible if $k$ is growing at the same time with $n$.
   Indeed, the kernel of an $(n-k) \times n$ Bernoulli random matrix contains the vector $(1 \etc 1)$ with probability $\left(c/\sqrt{n}\right)^n= \exp(-c' n \log n)$.
   It can also contain numerous other vectors of the same type with a similar probability.
   Hence the kernel of such matrix contains incompressible vectors with rigid arithmetic structure for $k = \Omega(\log n)$ which includes the range important for the question of Feige and Lellouche.
   
   Fortunately, the complete absence of vectors with a rigid arithmetic structure in the kernel is not necessary for proving a bound on the probability of a low rank.
   It is sufficient to rule out the situation where such vectors occupy a significant part of the kernel. 
   More precisely, we show that if $B$ is an $(n-k) \times n$ random matrix with i.i.d. subgaussian entries, then with probability at least $1-\exp(-ckn)$, its kernel contains a $(k/2)$-dimensional subspace free of the vectors with a rigid arithmetic structure. 
   Checking this fact is the main technical step in proving our main theorem.
   
   We outline the argument leading to it below. 
   We try to follow the geometric method developed in \cite{R square}, \cite{RV invertibility}.
   However, the aim of obtaining a super-exponential probability bound forces us to work with systems of problematic vectors instead of single ones.
   To handle such systems, we introduce a notion of an \emph{almost orthogonal} $l$-tuple of vectors in Section \ref{sec: Preliminary}. These systems are sufficiently simple to allow efficient estimates. At the same time, we show in Lemma \ref{lem: min config} that a linear subspace containing many ``bad'' vectors contains an almost orthogonal system of such vectors possessing an important minimality property. 
   
   Following the general scheme, we split the unit sphere of $\R^n$ into compressible and incompressible parts. 
   Let us introduce the respective definitions.
   \begin{definition}
   	Let $s \in [n]$ and let $\tau>0$. Define the set of $s$-sparse vectors by
   	\[
   	 \Sparse(s) = \{x \in \R^n: \ |\text{\rm supp}(x)| \le s\}
   	\]
   	and the sets of compressible and incompressible vectors by 
   	\begin{align*}
   	 \Comp(s,\tau)
   	 &=\{x \in S^{n-1}: \ \dist(x, \Sparse(s)) \le \tau \}, \\
   	  \Incomp(s,\tau)
   	  &= S^{n-1} \setminus \Comp(s,\tau).
   	\end{align*}
   \end{definition}
    Note that we define the sparse vectors in $\R^n$ and not in $S^{n-1}$. 
    This is not important but allows to shorten some future calculations.

   In Section \ref{sec: Compressible}, we show that the probability that the kernel of the matrix $B=(A_{[n-k] \times [n]})^\top$ contains an almost orthogonal system of $k/4$ compressible vectors is negligible.
   This is done by using a net argument, i.e., by approximating vectors from our system by vectors from a certain net. 
   The net will be a part of a scaled integer lattice, and the approximation will be performed by \emph{random rounding}, a technique widely used in computer science and introduced in random matrix theory by Livshyts \cite{Liv}.
   Let $B$ be a random matrix.
   The general net argument relies on obtaining a uniform lower bound for $\norm{By}_2$ over all points $y$ in the net and approximating a given point $x$ by the points of the net. 
   In this case, one can use the triangle inequality to obtain
   \[
     \norm{Bx}_2 \ge \norm{By}_2 - \norm{B} \cdot \norm{x-y}_2.
   \]
   This approach runs into problems in the absence of a good control of $\norm{B}$.
   However, if the net is constructed  a part of a scaled integer lattice, then one can choose the approximating point $y$ as a random vertex of the cubic cell containing $x$.
   This essentially allows to replace $\norm{B}$ in the approximation above by a more stable quantity $\norm{B}_{\HS}/\sqrt{n}$.
   Moreover, this replacement will be possible for a randomly chosen $y$ with probability close to $1$.

   In our case, we have to approximate the entire system of vectors while preserving the almost orthogonality property. 
   This makes the situation more delicate, and we can only prove that this approximation succeeds with probability which is exponentially small in $k$. 
   Fortunately, this is enough since we need just one approximation, so any positive probability is sufficient.
   
   In Section \ref{sec: Incompressible}, we assume that the kernel of $B$ contains a subspace of dimension $(3/4)k$ consisting of incompressible vectors and prove that with high probability, this subspace contains a further one of dimension $k/2$ which has no vectors with a rigid arithmetic structure. 
   The arithmetic structure is measured in terms of the \emph{least common denominator} (LCD) which is defined in Section \ref{subsec: LCD}.
   To this end, we consider a minimal almost orthogonal system of $k/4$ vectors having sub-exponential LCD-s and show that the presence of such system in the kernel is unlikely using the net argument and random rounding. 
   This is more involved than the case of compressible vectors since the magnitude of the LCD varies from $O(\sqrt{n})$ to the exponential level, and thus requires approximation on different scales. 
   To implement it, we decompose the set of such systems according to the magnitudes of the LCD-s and 
   then we scale each system by the sequence of its LCD-s.
    Because of the multiplicity of scales, the approximation has to satisfy a number of conditions at once. 
    At this step we also rely on random rounding allowing to check all the required conditions probabilistically.
   Verification that all of them can be satisfied simultaneously, although with an exponentially small probability performed in the proof in Lemma \ref{lem: approx} is the most technical part of the argument. 
   
   Finally, in Section \ref{sec: Rank}, we  collect all the ingredients and finish the proof of Theorem \ref{thm: rank}.

\section{Preliminary results} \label{sec: Preliminary}
 \subsection{Almost orthogonal systems of vectors}
 We will have to control the arithmetic structure of the subspace spanned by $n-k$ columns of the matrix $A$ throughout the proof.
 This structure is defined by the presence of vectors which are close to the integer lattice.
 To be able to estimate the probability that many such vectors lie in the subspace, we will consider special configurations of almost orthogonal vectors which are easier to analyze.
 This leads us to the following definition.
 \begin{definition} \label{def: nu-orthogonal}
 	Let $\nu \in (0,1)$.
 	An $l$-tuple of vectors $(v_1 \etc v_l) \subset \R^n \setminus \{0\}$ is called $\nu$-almost orthogonal if the $n \times l$ matrix $W$ with columns $\left(\frac{v_1}{\norm{v_1}_2} \etc \frac{v_l}{\norm{v_l}_2}  \right)$ satisfies
 	\[
 	1-\nu \le s_l(W) \le s_1(W) \le 1+\nu.
 	\]
 \end{definition}

Estimating the largest and especially the smallest singular values of a general deterministic matrix is a delicate task.
We employ a very crude criterion below.
 
 \begin{lemma}  \label{lem: almost ort}
 	Let $\nu \in [0, \frac{1}{4}]$ and let $(v_1 \etc v_l) \subset \R^n  \setminus \{0\}$ be a an $l$-tuple such that 
 	\[
 	 \norm{P_{\Span (v_1 \etc v_j)} v_{j+1}}_2
 	 \le \frac{\nu}{\sqrt{l}} \norm{v_{j+1}}_2
 	 \quad \text{for all } j \in [l-1].
 	\]
 	Then $(v_1 \etc v_l) \subset \R^l$ is a $(2\nu)$-almost orthogonal system. 
 	Moreover, if  $V$ is the $n \times l$ matrix with columns $v_1 \etc v_l$,
 	then 
 	\[
 	\Det^{1/2} (V^\top V) \ge 2^{-l} \prod_{j=1}^l \norm{v_j}_2.
 	\]
 \end{lemma}
 \begin{proof}
   Construct an orthonormal system in $\R^n$ by setting
   \[
     e_1= \frac{v_1}{\norm{v_1}_2},
     \qquad e_{j+1}= \frac{P_{(\Span (v_1 \etc v_j))^{\perp}} v_{j+1}}{\norm{P_{(\Span (v_1 \etc v_j))^\perp} v_{j+1}}_2} \quad \text{for all } j \in [l-1]
  \]
  and complete it to an orthonormal basis. The $n \times l$ matrix $W$ with columns $\Col_j(W)=\frac{v_j}{\norm{v_j}_2}$ written in this basis has the form $W= \begin{bmatrix}
  	\bar{W} \\0
  \end{bmatrix}$, where $W_u$ is an $l \times l$ upper triangular matrix. 
 	The assumption of the lemma yields
 	\[
      \left(  \sum_{i=1}^{j-1} \bar{W}_{i,j}^2 \right)^{1/2}
       =
       \norm{P_{\Span (v_1 \etc v_{j-1})} \Col_{j}(\bar{W})}_2 \le \frac{\nu}{\sqrt{l}} \quad \text{for all } j \in \{2 \etc l\}. 
   \] 	
 	and so,
   \[
     \sqrt{1-\frac{\nu^2}{l}} \le \bar{W}_{j,j} \le 1 \quad \text{for all } j \in [l],
   \]
   since $\norm{\Col_{j}(\bar{W})}_2=1$.
 	Therefore,
 	\[
 	  \norm{\bar{W}- \diag(\bar{W})}
 	  \le \norm{\bar{W}- \diag(\bar{W})}_{\HS}
 	  = \left( \sum_{j=1}^l \sum_{i<j} \bar{W}_{i,j}^2 \right)^{1/2}
 	  \le \nu,
 	\]
 	and thus
 	\begin{align*}
 	1-2\nu &\le 1 - \norm{I_l-\diag(\bar{W})} - \norm{\diag(\bar{W})-\bar{W}} \\
 	&\le s_l(\bar{W}) \le s_1(\bar{W}) \\
 	&\le 1 +  \norm{I_l-\diag(\bar{W})} + \norm{\diag(\bar{W})-\bar{W}} \\
 	&\le 1+ 2\nu.
 	\end{align*}
 	This implies the first claim of the lemma.
 	The second claim  immediately follows from the first one.
 \end{proof}

 The next lemma shows that if $W \subset \R^n \setminus \{0\}$ is a closed set and $E \subset \R^n$ is a linear subspace, then we can find a large almost orthogonal system in $E \cap W$ having a certain minimality property or a further linear subspace $F \subset E$ of a large dimension disjoint from $W$.
 This minimality property will be a key to estimating the least common denominator below.

\begin{lemma}[Almost orthogonal system] \label{lem: min config}
	There exists a constant $c>0$ for which one of the following holds.
	Let  $W \subset \R^n \setminus \{0\}$ be a closed set.  Let $l < k \le n$, and let $E \subset \R^n$ be a linear subspace of dimension $k$.
	Then at least one of the following holds.
	\begin{enumerate}
		\item \label{it: ort system} There exist vectors $v_1 \etc v_l \in E \cap W$ such that
		\begin{enumerate}
			\item \label{it: 13} The $l$-tuple $(v_1 \etc v_l)$ is $\left(\frac{1}{8}\right)$-almost orthogonal;
			\item \label{it: 14} For any $\theta \in \R^l$ such that $ \norm{\theta}_2 \le \frac{1}{20 \sqrt{l}}$ 
			\[
			\sum_{i=1}^{l} \theta_i v_i \notin W.
			\]  			
		\end{enumerate}
		\item \label{it: subspace} There exists a subspace $F \subset E$ of dimension $k-l$ such that  $F \cap W = \varnothing$.
	\end{enumerate}
\end{lemma}

\begin{remark} \label{rem: nu-orthogonal}
	The restriction $k \le c n^{1/2}$ in the formulation of Theorem \ref{thm: rank} stems from the condition $ \norm{\theta}_2 \le \frac{1}{20 \sqrt{l}}$ in Lemma \ref{lem: min config} \eqref{it: 14} which in turn arises from using Lemma \ref{lem: almost ort} for the $l$-tuple $v_1 \etc v_l$.
\end{remark}

\begin{proof}
	Let us construct a sequence of vectors $v_1 \etc v_{l'}, \ l' \le l$ with  $ \norm{v_1}_2 \le \norm{v_2}_2 \le \cdots \le \norm{v_{l'}}_2$ by induction.
	If $E \cap W = \varnothing$,  then \eqref{it: subspace} holds for any subspace $F$ of $E$ of dimension $k-l$, so the lemma is proved.
	Assume that 
	 $E \cap W \neq \varnothing$, and define $v_1$ as the vector of this set having the smallest norm.
	
	Let $2 \le j \le l-1$.
	For convenience, denote $v_0=0$.
		Assume that $j \in [l-1]$ and the vectors $v_1 \etc v_{j}$ with  $ \norm{v_1}_2 \le \norm{v_2}_2 \le \cdots \le \norm{v_j}_2$ and such that for all   $0 \le i \le j-1$, $v_{i+1}$ is the vector of the smallest norm in $E \cap W$ for which the inequality
	\[
	\norm{P_{\Span(v_0 \etc v_i) } v_{i+1}}_2 \le \frac{1}{16 \sqrt{l}} \norm{v_i}_2  
	\]
	holds.
	Note that if $j=1$, then the condition above is vacuous, and the vector $v_1$ has been already constructed.
	Assume that $j \ge 2$, and we have found such vectors $v_1 \etc v_j$.
	Consider the set
	\[
	H_j= \{ v \in E \cap W: \  \norm{P_{\Span(v_0 \etc v_j) } v}_2 \le \frac{1}{16 \sqrt{l}} \norm{v_j}_2  \}.
	\]
	If $H_j= \varnothing$, then \eqref{it: subspace} holds for any subspace of $E \cap (\Span(v_1 \etc v_j))^\perp$ of dimension $k-l$ which proves the lemma in this case.
	Otherwise, choose a vector $v \in H_j$  having the smallest norm and denote it by $v_{j+1}$.
	By construction, $\norm{v_{j+1}}_2 \ge \norm{v_j}_2$ since otherwise it would have been chosen at one of the previous steps.
	 
	Assume that we have run this process for $l$ steps and constructed such sequence $v_1 \etc v_l$.
	Then  for any $j \in [l]$ 
	\[
	\norm{P_{\Span(v_1 \etc v_{j-1})} v_j}_2  \le \frac{1}{16 \sqrt{l}} \norm{v_{j-1}}_2 
	\le  \frac{1}{16 \sqrt{l}} \norm{v_{j}}_2,
	\]
	and Lemma \ref{lem: almost ort} ensures that \eqref{it: 13} holds. 
	Therefore, to complete the induction step, we have to check only \eqref{it: 14}.
	Assume that there exists $\theta \in \R^{j+1}$ such that $ \norm{\theta}_2 \le \frac{1}{20 \sqrt{l}}$ and 
	\begin{equation} \label{eq: dis rho}
		\sum_{i=1}^{j+1} \theta_i v_i \in W.
	\end{equation}
	Let $V^{j}$ be the $n \times j$ matrix with columns $v_1 \etc v_j$. The already verified condition \eqref{it: 13} yields $\norm{V^{j}} \le \frac{9}{8}  \max_{i \in [j]} \norm{v_i}_2 \le \frac{9}{8}  \norm{v_{j}}_2$.
	Since $v_{j+1} \in H_j$,
	\begin{align*}
		\norm{P_{\Span(v_1 \etc v_j) } \left(\sum_{i=1}^{j+1} \theta_i v_i \right)}_2 
		&\le \norm{\sum_{i=1}^j \theta_i v_i}_2 + |\theta_{j+1}| \cdot  \norm{P_{\Span(v_1 \etc v_j) } v_{j+1}}_2 \\
		&\le  \norm{V^j} \cdot \norm{\theta}_2 + \norm{\theta}_2\cdot \frac{1}{16 \sqrt{l}} \norm{v_j}_2 \\
		&\le \left(\frac{9}{8}+ \frac{1}{16 \sqrt{l}}\right) \norm{\theta}_2\cdot  \norm{v_j}_2 \\
		&< \frac{1}{16 \sqrt{l}}  \norm{v_j}_2.
	\end{align*}
    The last inequality above uses that $\norm{\theta}_2 \le \frac{1}{20 \sqrt{l}}$.
	Since by the inductive construction, $v_{j+1}$ is the vector of the smallest norm in $H_j$ having this property, $\norm{\sum_{i=1}^{j+1} \theta_i v_i }_2 \ge \norm{v_{j+1}}_2$.
	On the other hand, by  \eqref{it: 13} and Lemma \ref{lem: almost ort}, $\norm{V^{j+1}} \le \frac{9}{8}  \norm{v_{j+1}}_2$,  so
	\[
	\norm{\sum_{i=1}^{j+1} \theta_i v_i}_2 
	\le \norm{V^{j+1}} \cdot \norm{\theta}_2
	\le \frac{9}{8} \norm{v_{j+1}}_2 \cdot \norm{\theta}_2 \le \frac{1}{16 \sqrt{l}}  \norm{v_{j+1}}_2 .
	\] 	
	This contradiction shows that \eqref{eq: dis rho} is not satisfied, so \eqref{it: 14} holds.		
\end{proof}

 \subsection{Concentration and tensorization}
 We will need several elementary concentration results. 
 To formulate them, we introduce a few definitions.
	Denote by $\LL(X,t)$ the Levy concentration function of a random vector $X \in \R^m$:
\[
\LL(X,t)= \sup_{y \in \R^m} \P(\norm{X-y}_2 \le t).
\]
Let $\xi \in \R$ be a random variable.
We will call it subgaussian if $\E \exp(\lambda |\xi|^2) < \infty$ for some $\lambda >0$ and denote
\[
 \norm{\xi}_{\psi_2}:= \inf \left\{s>0: \ \E \left[ \exp \left(\frac{|\xi|}{s}\right)^2 \right] \le 2 \right\}.
\]

 For technical reasons, let us restrict the class of random entries of the matrix and introduce some parameters controlling their behavior.
 First, without loss of generality, we may assume that the entries of $A$ are centered, i.e., $\E a_{i,j}=0$. Indeed, since all entries are i.i.d., subtracting the expectation from each one results in a rank one perturbation of the matrix $A$ which does not affect the conclusion of Theorem \ref{thm: rank}.
 Second, since the entries are non-constant, $\LL(a_{i,j},t) <1$ for some $t>0$. After an appropriate scaling the entries, we can assume that $t=1$.
 Therefore, throughout the paper, we will assume that the entries of the matrix $A$ are i.i.d. copies of a random variable $\xi$ satisfying the following conditions:
\begin{equation} \label{eq: xi}
	\E \xi=0, \quad \norm{\xi}_{\psi_2} \le K, \quad \LL(\xi,1) \le 1-p.
\end{equation}
Without loss of generality, we may assume that $K \ge 1$.

 Throughout the paper we consider random matrices whose entries are independent copies of a random variable $\xi$ satisfying \eqref{eq: xi}. 
 The constants $c,C, C'$ etc. appearing in various formulas below may depend on $p$ and $\norm{\xi}_{\psi_2}$.
 
 \begin{lemma}[Operator norm] \label{lem: operator norm}
 	Let $m \le n$, and let $Q$ be an $m \times n$ matrix with centered  independent entries $q_{i,j}$ such that $|q_{i,j}| \le 1$.  
 	Then
 	\[
 	 \P ( \norm{Q} \ge C_{\ref{lem: operator norm}} \sqrt{n}) \le \exp(-c_{\ref{lem: operator norm}} n).
 	\]
 \end{lemma}

  Lemma \ref{lem: operator norm} follows from a general norm estimate for a random matrix with centered subgaussian entries, see, e.g. \cite{RV invertibility}. 
  It is easy to see that the statement of the lemma is optimal up to constants $C,c$.
  Note that the event that $ \norm{Q} \ge C \sqrt{n}$ has probability which is exponentially small in $n$. 
  Such bound is sufficient for the application we have in mind, but is not strong enough to bound the operator norm of $A$. 
  Indeed, as our aim is to prove the bound $\exp(-ckn)$ for the probability that the rank of $A$ is smaller than $n-k$ and $k$ can be large, we cannot exclude events of probability $\exp(-cn)$.
  This forces us to consider another matrix norm which enjoys stronger concentration properties. 
  
  For a matrix with subgaussian entries, we probe a stronger bound for the Hilbert-Schmidt norm.
 
 \begin{lemma}[Hilbert-Schmidt norm]  \label{lem: HS norm}
 	Let $m \le n$ and let $A$ be an $m \times n$ matrix whose entries are independent copies of a random variable $\xi$ satisfying \eqref{eq: xi}.
 	Then
 	\[
 	  \P (\norm{A}_{\HS} \ge 2 K n)
 	  \le \exp ( - c n^2).
 	\]
 \end{lemma}
\begin{proof}
	Since $\E \xi^2 \le \norm{\xi}_{\psi_2}< \infty$,
	\[
	  \E \exp  \left( \frac{|\xi^2 - \E \xi^2|}{\norm{\xi}_{\psi_2}^2}\right)
	  \le \E \exp  \left(\frac{  \xi^2}{\norm{\xi}_{\psi_2}^2} +1\right) 
	  \le 2e,
	\]
	which shows that $Y=\xi^2 - \E \xi^2$ is a centered sub-exponential random variable.
	Taking into account that
	\[
	 \sum_{i=1}^m \sum_{j=1}^n  \E a_{i,j}^2 \le K^2 n^2
	\]
	and Bernstein's inequality \cite{V HDP}, we obtain
	\[
	 \P  (\norm{A}_{\HS} \ge 2K n)
	 =  \P  \left(\sum_{i=1}^m \sum_{j=1}^n (a_{i,j}^2 - \E a_{i,j}^2)\ge 3 K^2 n^2 \right)
	 \le \exp(-c  n^2)
	\]
	as required.
\end{proof}

 We will also need a tensorization lemma for the small ball probability similar to Lemma 2.2 \cite{RV invertibility}.
 \begin{lemma}[Tenzorization] \label{lem: tenzorization}
 	Let $m, M>0$ and let $Y_1 \etc Y_n \ge 0$ be independent random variables such that $\P (Y_j \le s) \le (Ms)^m$ for all $s \ge s_0$.
 	Then
 	\[
 	  \P \left( \sum_{j=1}^n Y_j \le n t \right)
 	  \le (C M  t)^{mn}
 	  \quad \text{for all } t \ge  s_0.
 	\]
 \end{lemma}

 \begin{proof}
 	Let $t \ge  s_0$. By Markov's inequality,
 	\begin{align*} 
 		 \P \left( \sum_{j=1}^n Y_j \le n t \right)
 		&\le  \E \left[ \exp \left( m n - \frac{m}{t}\sum_{j=1}^n Y_j \right) \right] \\
 		&= e^{m n} \prod_{j=1}^{n} \E \exp \left(-\frac{m}{t} Y_j\right),	
 	\end{align*}
  where
  \begin{align*} 
  	 \E \exp \left(-\frac{m}{t} Y_j\right)
  	 &=	\int_0^1 \P \left[  \exp \left(-\frac{m}{t} Y_j\right) >s \right] \, ds
  	 = \int_0^{\infty} e^{-u} \P \left[ Y_j < \frac{t}{m} u \right] \, du \\
  	 &\le  \int_0^{m} e^{-u} \P \left[ Y_j < \frac{t}{m} \right] \, du 
  	 +  \int_m^{\infty} e^{-u} \P \left[ Y_j < \frac{t}{m} u \right] \, du \\
  	 & \le \left( M t\right)^m +  \int_m^{\infty} e^{-u} \left(\frac{M t}{m} u\right)^m  \, du \\
  	 &\le \left( 1+ \frac{\Gamma(m+1)}{m^m} \right) \cdot \left(M t\right)^m 
  	 \le (CM t)^m.
  	\end{align*}
  Here we used that $\P \left[ Y_j < \frac{t}{m} u \right] \le \P \left[ Y_j < \frac{t}{m}  \right]$ for $u \in (0,1)$ in the first inequality and the Stirling formula in the last one.
  Combining the two inequalities above completes the proof.
 \end{proof}
 
\subsection{Least common denominators and the small ball probability}	\label{subsec: LCD}
 The least common denominator (LCD) of a sequence of real numbers originally introduced in \cite{RV invertibility} turned out to be a useful tool to gauge the behavior of the Levy concentration function of a linear combination of independent random variables with constant coefficients. Its various versions played a crucial role in proving quantitative estimates of invertibility of random matrices, see e.g. \cite{R invertibility survey} and the references therein as well as more recent works including \cite{V symmetric}, \cite{BR invertibility}, \cite{LTV}, \cite{CJMS singular}, and numerous other papers.
 In what follows, we use the extension of the LCD to matrices introduced in \cite{RV no-gaps}.

\begin{definition}
	Let $V$ be an $m \times n$ matrix, and let $L>0, \a \in (0,1]$. 
	Define the \emph{least common denominator} (LCD) of $V$ by
	\[
	  D_{L,\a}(V)
	  = \inf \left( \norm{\theta}_2: \theta \in \R^m, \ 
	     \dist(V^\top \theta, \Z^n) < L \sqrt{\log_+ \frac{\a \norm{V^\top \theta}_2}{L} }  \right).
	\]
	If $E \subset \R^n$ is a linear subspace, we can adapt this definition to the orthogonal projection $P_E$ on $E$ setting
	\[
       D_{L,\a}(E)=D_{L,\a}(P_E)
       = \inf \left( \norm{y}_2: y \in E, \ 
         \dist(y, \Z^n) < L \sqrt{\log_+ \frac{\a \norm{y}_2}{L} }  \right).
    \]	
\end{definition}
This is a modification of \cite[Definition 6.1]{V symmetric} and \cite[Definition 7.1]{RV no-gaps}, where the same notion was introduced with $\a=1$.

We will use the following concentration function estimate in terms of the LCD and its corollary proved in \cite{RV no-gaps}.

\begin{theorem}[Small ball probabilities via LCD]			\label{thm: SBP LCD}
	Consider a random vector $\xi = (\xi_1,\ldots,\xi_n)$, where $\xi_k$ are i.i.d. copies
	of a real-valued random variable $\xi$ satisfying \eqref{eq: xi}.
	Consider a matrix $V \in \R^{m \times n}$. Then for every $L \ge \sqrt{m/p}$
	we have
	\begin{equation}							\label{eq: SBP LCD}
		\LL(V^\top \xi, t \sqrt{m}) 
		\le \frac{\big(CL/(\a \sqrt{m})\big)^m}{\det(VV^\top)^{1/2}}
		\left( t + \frac{\sqrt{m}}{D_{L, \a}(V^\top)} \right)^m, \quad t \ge 0.
	\end{equation}
\end{theorem}

Theorem \ref{thm: SBP LCD} with $\a=1$ is \cite[Theorem 7.5]{RV no-gaps}. We notice that exactly the same proof with the modified definition of the LCD yields Theorem \ref{thm: SBP LCD} with a general $\a$.

\begin{corollary}[Small ball probabilities for projections]					\label{cor: SBP proj}
	Consider a random vector $\xi = (\xi_1,\ldots,\xi_N)$, where $\xi_k$ are i.i.d. copies
	of a real-valued random variable $\xi$ satisfying \eqref{eq: xi}.
	Let $E$ be a subspace of $\R^N$ with $\dim(E) = m$, and
	let $P_E$ denote the orthogonal projection onto $E$.
	Then for every $L \ge \sqrt{m/p}$ we have
	\begin{equation}							\label{eq: SBP proj}
		\LL(P_E \xi, t \sqrt{m}) \le \left( \frac{CL}{\a \sqrt{m}} \right)^m
		\left( t + \frac{\sqrt{m}}{D_{L,\a}(E)} \right)^m, \quad t \ge 0.
	\end{equation}
\end{corollary}

We will need a lemma which essentially generalizes the fact that the LCD of an incompressible vector is $\Omega(\sqrt{n})$.
We will formulate it in a somewhat more technical way required for the future applications.
\begin{lemma} \label{lem: LCD compressible}
	Let $s, \a \in (0,1)$.
	Let $U$ be an $n \times l$ matrix such that $U \R^l \cap S^{n-1} \subset \Incomp(sn, \a)$
	Then any $\theta \in \R^l$ with $\norm{U \theta}_2 \le \sqrt{sn}/2$ satisfies
	\[
	   \dist(U \theta, \Z^n) \ge L \sqrt{\log_+ \frac{\a \norm{U \theta}_2}{L} }.
	\]
\end{lemma}

\begin{proof}
	Take any $\theta \in \R^l$ such that $\norm{U \theta}_2 \le  \sqrt{sn}/2$. Let $x \in \Z^l$ be such that
	\[
	  \norm{U \theta -x}_2 = \dist(U \theta, \Z^n) \le  \norm{U \theta}_2.
	\]
	Then by the triangle inequality, $\norm{x}_2 \le  \sqrt{s n}$.
	Since the coordinates of $x$ are integer, this implies that $|\text{supp}(x)| \le  s n$.
	Therefore, 
	\[
	 \norm{\frac{U \theta}{\norm{U \theta}_2} - \frac{x}{\norm{U \theta}_2} }_2 \ge \a
	\]
	since $\frac{x}{\norm{U \theta}_2} \in \Sparse(s n)$.
	Combining the two previous inequalities, we see that 
	\[
	  \a \norm{U \theta}_2 \le \norm{U \theta -x}_2
	  = \dist(U \theta, \Z^n).
	\]
	The desired inequality follows now from an elementary estimate $t > \sqrt{\log_+ t}$ valid for all $t>0$ which is applied with $t=\frac{\a \norm{U \theta}_2}{L}$.
\end{proof}

\subsection{Integer points inside a ball}
  We will need a simple lemma estimating the number of integer points inside a ball in $\R^n$.
  Denote the Euclidean ball of radius $R$ centered at $0$ by $B(0,R)$ and the cardinality of a set $F$ by $|F|$.
  \begin{lemma} \label{lem: integer inside ball}
  	For any $R>0$,
  	\[
  	  	|\Z^n \cap B(0,R)| \le \left(2+ \frac{C R}{\sqrt{n}} \right)^n.
  	\] 
  \end{lemma}
The proof immediately follows by covering $B(0,R)$ by unit cubes and estimating the volume of their union.

\section{Compressible vectors} \label{sec: Compressible}

  The aim of this section is to prove that it is unlikely that the kernel of a rectangular matrix with i.i.d. entries satisfying \eqref{eq: xi}
  contains a large almost orthogonal system of compressible vectors. More precisely, we prove that the probability of such event does not exceed $\exp(-c l m)$, where $m$ is the number of rows of $B$ and $l$ is the number of vectors in the system. 
  The compressibility parameters will be selected in the process of the proof and after that, fixed for the rest of the paper.
  In Section \ref{sec: Rank}, we will apply this statement with $m=n-k$ and $l=k/4$, in which case the probability of existence of such almost orthogonal system becomes negligible for our purposes.
  
  We start with bounding the probability of presence of a fixed almost orthogonal system in the kernel of $B$. This bound relies on a corollary of the Hanson-Wright inequality, see \cite[Corollary 2.4]{RV Hanson-Wright}.
  Note that this result applies to any almost orthogonal system, not only to a compressible one.

\begin{lemma} \label{lem: individual}
	 Let $m \le n$, and let let $B$ be an $m \times n$ matrix whose entries are i.i.d. random variables satisfying \eqref{eq: xi}.
	Let $l \le n$ and let $v_1 \etc v_l \in S^{n-1}$ be an $l$-tuple of $\left(\frac{1}{2}\right)$-almost orthogonal vectors.
	Then
	\[
	\P \left( \norm{B v_j}_2 \le C_{\ref{lem: individual}}  \sqrt{m} \ \text{ for all } j \in [l] \right)
	\le \exp (-c_{\ref{lem: individual}}  l m).
	\]
\end{lemma}

\begin{proof}
	Let $V=(v_1 \etc v_l)$ be the $n \times l$ matrix formed by columns $v_1 \etc v_l$.
	The assumption of the lemma implies that $\norm{V} \le 2 \max_{j \in [l]} \norm{v_j}_2 =2$.
	On the other hand
	\[
	\sum_{j=1}^l s_j^2(V) = \norm{V}_{\HS}^2 = \sum_{j=1}^l \norm{v_j}_2^2 =l.
    \] 
    Note that if $\xi$ is a random variable satisfying \eqref{eq: xi}, then $\E \xi^2 \ge \P(|\xi| \ge 1) \ge p$.
    Let $\eta \in \R^n$ be a random vector with i.i.d. coordinates satisfying \eqref{eq: xi}.
	By the Hanson-Wright inequality,
	\[
	\P \left( \norm{V^\top \eta}_2^2 \le \frac{p}{2} l \right)
	\le \P \left( \norm{ V^\top \eta}_2^2 \le \frac{p}{2}  \norm{ V}_{\HS}^2 \right)
	\le \exp \left( -c  l\right). 
	\]
	Let $\eta_1 \etc \eta_m$ be i.i.d. copies of $\eta$.
	Then
	\[
	\P \left( \sum_{i=1}^m \norm{V^\top \eta_i}_2^2 \le \frac{p}{4}  l n\right)
	\le \exp \left(- \frac{c}{2}  l m \right).
	\]
	Indeed, the condition $ \sum_{i=1}^m \norm{V \eta_i}_2^2 \le \frac{p}{4}  l m$ implies that $ \norm{V \eta_i}_2^2 \le \frac{p}{2}  l $ for at least $m/2$ indexes $i \in [m]$. 
	These events are independent, and the probability of each one does not exceed $\exp(-cl)$.
	
	Applying the inequality above with $\eta_j=(\Row_j(B))^\top$, we obtain
	\begin{align*}
		\P \left( \norm{B v_j}_2^2 \le \frac{p}{4}  m \text{ for } j \in [l]\right)
		&\le \P \left( \sum_{j=1}^l \norm{B v_j}_2^2 \le \frac{p}{4}  m \cdot l \right) \\
		&= \P \left( \sum_{i=1}^m \norm{V^\top \eta_i}_2^2 \le \frac{p}{4} l m  \right)
		\le \exp (-\frac{c}{2}  l m).
	\end{align*}
	The proof is complete.
\end{proof}

  The next statement, Proposition \ref{prop: compressible}, contains the main result of this section. We will extend the bound of Lemma \ref{lem: individual} from the presence of a fixed almost orthogonal system of compressible vectors in the kernel of a random matrix to the presence of any such system. 

  The proof of Proposition \ref{prop: compressible} follows the  general roadmap of the geometric method.
  We start with constructing a special net for the set of compressible vectors in $S^{n-1}$. This net will consist of the vectors from a scaled copy of the integer lattice in $\R^n$. 
  The vectors of an almost orthogonal system will be then approximated by the vectors from this net using the procedure of \emph{random rounding}. 
  This procedure whose use  in random matrix theory was pioneered by Livshyts \cite{Liv} has now numerous applications in the problems related to invertibility. 
  One of its advantages is that it allows to bound the approximation error in terms of a highly concentrated Hilbert-Schmidt norm instead of the operator norm of the matrix. 
  In our case, this approximation presents two new special challenges. 
  First, we have to approximate all vectors $x_1 \etc x_l$ forming an almost orthogonal system at the same time and in a way that preserves almost orthogonality.
  Second, the vectors of the approximating system have to retain some sparsity properties the original vectors. 
  We will show below that all these requirements can be satisfied simultaneously for a randomly chosen approximation.
  The probability of that will be exponentially small in $l$ yet positive, which is sufficient since we need only to show the existence of such approximation.

\begin{proposition}  \label{prop: compressible}
	Let $k,n \in \N$ be such that $k \le n/2$ and let $B$ be an $(n-k) \times n$ matrix whose entries are i.i.d. random variables satisfying \eqref{eq: xi}.
	There exists $\tau >0$ such that  the probability that there exists a $\left( \frac{1}{4} \right) $-almost orthogonal $l$-tuple $ x_1 \etc x_l \in \Comp(\tau^2 n, \tau^4)$ with $l \le \tau^3 n$ and
	\[
	\norm{B x_j}_2 \le \tau \sqrt{n} \ \text{ for all } j \in [l]
	\]
	is less than $\exp(-c l n)$.
\end{proposition}

\begin{proof}
	Let $\tau \in (0,1/2)$ be a  number to be chosen later, and set
	\[
	T=  \left\{ v \in  \frac{\tau}{\sqrt{n}} \Z^n: \ \norm{v}_2 \in \left[\frac{1}{2}, 2\right] \right\}.
	\]
	Then Lemma \ref{lem: integer inside ball} applied to $R=\frac{\sqrt{n}}{\tau}$  yields
	\begin{equation}   \label{eq: card(T)}
		|T \cap \Sparse(4 \tau^2 n)|
		 \le \binom{n}{4 \tau^2 n} \cdot \left( 2+ \frac{C}{\tau}\right)^{4 \tau^2 n}
		\le \left( \frac{C' }{ \tau^3}\right)^{4 \tau^2 n}.
	\end{equation}
	
	Denote the coordinates of a vector $x \in \R^n$ by $x(1) \etc x(n)$.
	Consider a $\left(\frac{1}{4}\right)$-almost orthogonal $l$-tuple
	$x_1 \etc x_l \in \Comp(\tau^2 n, \tau^4)$.
	Since $x_j \in \Comp(\tau^2 n,   \tau^4)$, there is a set $I_1(j) \subset [n]$ with $|I_1(j)| \le \tau^2 n$ such that 
	\[
	\sum_{i \in [n] \setminus I_1(j)} x_j^2(i) \le \tau^8.
	\]
	Using an elementary counting argument, we conclude that there exists $I_2(j) \supset I_1(j)$ with $|I_2(j)| \le 2 \tau^2 n$ such that  
	\[
	  |x_i| \le \frac{ \tau^3 }{ \sqrt{n}}
	  \quad \text{for any }  i \in [n] \setminus I_2(j).
	\]
	
	For $j \in [l]$, define the vector $w_j=(w_j(1) \etc w_j(n))$ by
	\[
	w_j(i)=  \frac{\tau}{ \sqrt{n}} \cdot \left \lfloor   \frac{ \sqrt{n}}{\tau} |x_j(i)| \right \rfloor \sign (x_j(i)).
	\]
	This form of rounding is chosen to approximate small in the absolute value coordinates of $x_j$ by zeros.

	Define independent random variables $\e_{i,j}$ such that
	\begin{align*}
		\P(\e_{i,j}  = w_j(i)- x_j(i))
		&= 1-  \frac{ \sqrt{n}}{\tau} |x_j(i)-w_j(i)|,
		\intertext{ and } 
		\P \left(\e_{i,j}  = w_j(i)- x_j(i)+ \frac{\tau}{\sqrt{n}} \sign(x_j(i)) \right)
		&= \frac{\sqrt{n}}{\tau} |x_j(i)-w_j(i)|.
	\end{align*}
	Set 
	\[
	v_j=x_j + \sum_{i=1}^n \e_{i,j} e_i. 
	\]
	Then $v_j$ is a random vector such that $\E v_j=x_j$. 
	Moreover, $\norm{v_j-x_j}_2 \le \tau<1/2$, and so  $\norm{v_j}_2 \in [\frac{1}{2},2]$, which implies that  $v_j \in T$ for all $j \in [l]$.

	The definition of $v_j$ above means that for any $i \in [n] \setminus I_2(j)$, $w_j(i)=0$, and so  $\P (v_j(i) \neq 0)  \le \tau^2$ for these $i$.
	Set $I_3(j)=I_2(j) \cup \{ i \in [n] \setminus I_2(j) \ v_j(i) \neq 0 \}$.
	Since the events $v_j(i) \neq 0$ are independent for all $ i  \in [n] \setminus I_2(j)$ for a given $j \in [l]$, Chernoff's inequality in combination with the union bound over $j$ yield
	\begin{equation}  \label{eq: comp1}
		\P ( \forall j \in [l] \ |I_3(j)| \le 4 \tau^2 n) \ge 1- l \exp(-c \tau^2 n).
	\end{equation}
	Note that if $ |I_3(j)|< 4 \tau^2 n$ for all $j \in [l]$ then all the vectors $v_1 \etc v_l$ belong to $T \cap \Sparse(4 \tau^2 n)$.
	
	Let us form the $n \times l$ matrices $X$ and $V$ with columns $x_1 \etc x_l$ and $v_1 \etc v_l$ respectively. 
	Then the matrix $V-X$ has independent centered entries $\e_{i,j}$ whose absolute values are bounded by $\frac{\tau}{\sqrt{n}}$.
	This means that the random variables $\frac{\sqrt{n}}{\tau} \e_{i,j} $ satisfy the assumptions of Lemma \ref{lem: operator norm}.
	In view of this lemma,
	\[
	\P \left(\norm{V-X} \le C_{\ref{lem: operator norm}} \tau \right)
	\ge 1- \exp(-c_{\ref{lem: operator norm}} n).
	\]
	Define the diagonal matrix $D_V= \diag (\norm{v_1}_2 \etc \norm{v_l}_2)$.
	Recall that
	\[
	1- \tau \le \norm{v_j}_2 \le  1 + \tau
	\]
	for all $j \in [l]$,
	and $s_l(X) \ge \frac{3}{4}$ since the vectors $x_1 \etc x_l$ are $\left(\frac{1}{4}\right)$-almost orthogonal.
	Hence, if the event $\norm{V-X} \le  \tau$ occurs, then
	\begin{align*}
		s_l(V D_V^{-1}) 
		&\ge s_l(X D_V^{-1}) - \norm{X-V} \cdot \norm{D_V^{-1}} \\
		&\ge s_l(X) \cdot s_l( D_V^{-1}) - \norm{X-V} \cdot \norm{D_V^{-1}} \\
		&\ge \frac{3}{4} (1+\tau)^{-1} - C_{\ref{lem: operator norm}} \tau \cdot (1-\tau)^{-1} \\
		&\ge \frac{1}{2}
	\end{align*}
    where the last inequality holds if 
	\[
	\tau  \le \tau_0
	\]
	for some $\tau_0>0$.
	
	Similarly, we can show that $s_1(V D_V^{-1}) \le \frac{3}{2}$, thus proving that the vectors $v_1 \etc v_l$ are $\left(\frac{1}{2}\right)$-almost orthogonal.
	This shows that
	\begin{equation}  \label{eq: comp2}
		\P \left( v_1 \etc v_l \text{ are $\left( \frac{1}{2} \right)$-almost orthogonal}\right) 
		\ge 1- \exp(-c n).
	\end{equation}
	
	Let $\EE_{\HS}$ be the event that $\norm{B}_{\HS} \le 2Kn$.
    Lemma \ref{lem: HS norm} yields that 
    \[
      \P(\EE_{\HS}) \ge 1- \exp(-cn^2).
    \]
	Condition on a realization of the matrix $B$ such that $\EE_{\HS}$ occurs.
	Since the random variables $\e_{i,j}$ are independent, 
	\begin{align*}
		\E \norm{B (x_j-v_j)}_2^2 
		&= \E \norm{\sum_{i=1}^n \e_{i,j} B e_i}_2^2
		= \sum_{i=1}^n \E \e_{i,j}^2 \norm{B e_i}_2^2 \\
		&\le \left( \frac{\tau}{ \sqrt{n}} \right)^2 \norm{B}_{\HS}^2 \le 4K^2 \tau^2 n.
	\end{align*}
	Hence by Chebyshev's inequality
	\[
	\P \left[\norm{B (x_j-v_j)}_2 \le 3 K \tau \sqrt{n} \mid \EE_{\HS} \right]
	\ge \frac{1}{2}.
	\]
	In view of the independence of these events for different $j$,
	\begin{equation}  \label{eq: comp3}
		\P \left[ \forall j \in [l] \ \norm{B (x_j-v_j)}_2 \le 3 K \tau \sqrt{n} \mid \EE_{\HS} \right]
		\ge 2^{-l}.
	\end{equation}
	Let us summarize \eqref{eq: comp1}, \eqref{eq: comp2}, and \eqref{eq: comp3}.
	Recall that $l \le \tau^3 n$.
	If $\tau$ is sufficiently small, i.e., $\tau \le \tau_1$ for some $\tau_1>0$,
	then 
	\[
	1- \exp(-c_{\ref{lem: operator norm}} n) - l\exp(-c \tau^2 n)+ 2^{-l} >1.
	\]
	This means that
	conditionally on $B$ for which the event $\EE_{\HS}$ occurs, we can find a realization of random variables $\e_{i,j}, \ i\in [n], j \in [l]$ such that
	\begin{itemize}
		\item the vectors $v_1 \etc v_l$ are  $\left( \frac{1}{2} \right)$-almost orthogonal;
		\item  $v_1 \etc v_l \in T \cap \Sparse(4 \tau^2 n)$, and
		\item $ \norm{B (x_j-v_j)}_2 \le 3 K \tau \sqrt{n}$  for all $j \in [l]$.
	\end{itemize}

	Assume that there exists a $\left(\frac{1}{4}\right)$-almost orthogonal $l$-tuple
	$x_1 \etc x_l \in \Comp(\tau^2 n, \tau^4)$ such that 
	$\norm{B x_j}_2 \le \tau \sqrt{n} \ \text{ for all } j \in [l]$.
	Then the above argument shows that conditionally on $B$ such that $\EE_{\HS}$ occurs, we can find vectors $v_1 \etc v_l \in   T \cap \Sparse(4 \tau^2 n)$ which are  $\left( \frac{1}{2} \right)$-almost orthogonal such that  $ \norm{B v_j}_2 \le 4 K \tau \sqrt{n}$  for all $j \in [l]$ since $K \ge 1$.
	Therefore,
	\begin{align*}
		&\P (\exists x_1 \etc x_l \in \Comp(\tau^2 n, \tau^4): \ \norm{B x_j}_2 \le \tau \sqrt{n} \ \text{ for all } j \in [l] \text{ and } \EE_{\HS} ) \\
		\le &\P \Big[\exists v_1 \etc v_l \in   T \cap \Sparse(4 \tau^2 n): \ v_1 \etc v_l \text{ are } \left( \frac{1}{2} \right)\text{-almost orthogonal } \\
		&\qquad \qquad \text{and }\norm{B v_j}_2 \le 4 K \tau \sqrt{n} \ \text{ for all } j \in [l] \mid \EE_{\HS} \Big] \cdot \P(\EE_{\HS}) \\
		= &\P \Big(\exists v_1 \etc v_l \in   T \cap \Sparse(4 \tau^2 n): \ v_1 \etc v_l \text{ are } \left( \frac{1}{2} \right)\text{-almost orthogonal } \\
		&\qquad \qquad \text{and }\norm{B v_j}_2 \le 4 K \tau \sqrt{n} \ \text{ for all } j \in [l] \text{ and } \EE_{\HS} \Big).
	\end{align*}
	Let us show that the latter probability is small. 
	Assume that 
	\[
	  \tau 
	  \le \min \left( \tau_0, \tau_1 ,\frac{C_{\ref{lem: individual}}}{4K} \right).
	\]
	Recall that the number of rows of $B$ satisfies $n-k \ge n/2$.
	In view of Lemma \ref{lem: individual} and \eqref{eq: card(T)}, 
	\begin{align*}
		&\P \big( \exists v_1 \etc v_l \in T \cap \Sparse(4 \tau^2 n):  \  (v_1 \etc v_l) \text{ is $\left( \frac{1}{2}\right)$-almost orthogonal} \\
		&\qquad  \text{ and }  \norm{B v_j}_2 \le  4K \tau  \sqrt{n} \ \text{ for all } j \in [l] \big) \\
		&\le \left( \frac{C }{ \tau^3}\right)^{4 \tau^2 n \cdot l} \cdot \exp \left(-\frac{c_{\ref{lem: individual}} }{2}l n \right)
		\le \exp \left( - \left[\frac{ c_{\ref{lem: individual}} }{2}-4 \tau^2 \log \left( \frac{C  }{\tau^3}  \right)\right]  l n \right) \\
		&\le \exp \left(- \frac{c_{\ref{lem: individual}}}{4} l n \right),
	\end{align*}
	where the last inequality holds if we choose $\tau$ sufficiently small.

	The previous proof shows that
	\begin{align*}
		&\P (\exists x_1 \etc x_l \in \Comp(\tau^2 n, \tau^4) \ \norm{B x_j}_2 \le \tau \sqrt{n} \ \text{ for all } j \in [l] \text{ and } \EE_{\HS} ) \\
		\le &\exp \left(- \frac{c_{\ref{lem: individual}}}{4} l n \right).
	\end{align*}
	In combination with the inequality $\P(\EE_{\HS}^c) \le \exp(-c n^2)$, this completes the proof.
\end{proof}
  We will fix the value of $\tau$ for which Proposition \ref{prop: compressible} holds for the rest of the paper.

\section{Incompressible vectors} \label{sec: Incompressible}

 The main statement of this section, Proposition \ref{prop: matrix V}  bounds the probability that the kernel of a rectangular matrix $B$ with   i.i.d. entries satisfying assumptions \eqref{eq: xi} contains an almost orthogonal system of incompressible vectors with  subexponential common denominators. In what follows, $B$ will be the $(n-k) \times n$ matrix whose rows are $\Col_1(A)^\top \etc \Col_{n-k}(A)^\top,$ and the required probability should be of order $\exp(-c kn)$ to fit Theorem \ref{thm: rank}.
 The need to achieve such a tight probability estimate requires considering the event that $l$ vectors in the kernel of $B$ have subexponential least common denominator. The number $l$ here is proportional to $k$.
 Recall that a vector has a relatively small least common denominator if after being scaled by a moderate factor, it becomes close to the integer lattice. 
 Since we have to consider $l$ such vectors at once, and the norms of these scaled copies vary significantly, it is more convenient to consider these copies, and not the original unit vectors as we did in Proposition \ref{prop: compressible}.
 Moreover, to bound the probability, we have to consider all vectors with a moderate least common denominator in the linear span of the original system of $l$ vectors in the kernel of $B$.
 To make the analysis of such linear span more manageable, we will restrict our attention to the almost orthogonal systems. 
 This restriction will be later justified by using Lemma \ref{lem: min config}.
 
  Throughout the paper, we set 
  \begin{equation} \label{eq: def L}
	L= \sqrt{k/p},  \quad \a=\frac{\tau^4}{4}.
  \end{equation}
  where $k$ appears in Theorem \ref{thm: rank},  $p$ is a parameter from \eqref{eq: xi}, and $\tau$ was chosen at the end of Section \ref{sec: Compressible}.

\begin{proposition} \label{prop: matrix V}
	Let $\rho \in (0,\rho_0)$, where $\rho_0=\rho_0(\tau)$ is some positive number.
	Assume that $l \le k \le  \frac{\rho}{2}  \sqrt{n}$.

	Let $B$ be an $(n-k) \times n$ matrix with i.i.d. entries satisfying \eqref{eq: xi}. 
	Consider the event $\EE_l$  that there exist vectors $v_1 \etc v_l \in \ker(B)$ having the following properties.
	\begin{enumerate}
		\item \label{it: 1} $\frac{\tau}{8} \sqrt{n} \le \norm{v_j}_2 \le \exp \left( \frac{\rho^2 n}{4 L^2}\right)$ for all $j \in [l]$;
		\item \label{it: 15} $\Span (v_1 \etc v_l) \cap S^{n-1} \subset \Incomp (\tau^2 n, \tau^4)$;
		\item \label{it: 2} The vectors $v_1 \etc v_l$ are $\left(\frac{1}{8}\right)$-almost orthogonal;
		\item \label{it: 3} $\dist(v_j, \Z^n) \le \rho \sqrt{n}$ for $j \in [l]$;
		\item \label{it: 4} 
		The  $n \times l$ matrix $V$  with columns $v_1 \etc v_l$
		satisfies
		\[
		 \dist (V \theta, \Z^n) > \rho \sqrt{n}
		\]
		for all $\theta \in \R^l$ such that $\norm{\theta}_2 \le \frac{1}{20 \sqrt{l}}$ and $\norm{V \theta}_2 \ge \frac{\tau}{8} \sqrt{n}$.
	\end{enumerate}
    Then
    \[
      \P (\EE_l) \le \exp \left(- l n \right).
    \]
\end{proposition}

Conditions \eqref{it: 1}-\eqref{it: 3} mean that $v_1 \etc v_l$ is a $(1/8)$-almost orthogonal system of incompressible vectors close to the integer lattice, and condition \ref{it: 4} is a certain minimality property of this system.

 To simplify the analysis, we will  tighten condition \eqref{it: 1} of Proposition \ref{prop: matrix V}   restricting the magnitudes of the norms of $v_j$ to some dyadic intervals. 
 Denote for shortness 
 \begin{equation} \label{eq: def r,R}
   r= \frac{\tau}{16}, \quad R= \exp \left(\frac{\rho^2 n}{4 L^2}\right).
 \end{equation}
  Consider a vector $\mathbf{d}=(d_1 \etc d_l) \in [r \sqrt{n}, R]^l$ and define the set  $W_{\mathbf{d}}$ be the set of $l$-tuples of vectors $v_1 \etc v_l \in \R^n$ satisfying
 \[
\norm{v_j}_2 \in [d_j, 2d_j] \text{ for all } j \in [l];
\]
 and conditions \eqref{it: 15} -- \eqref{it: 4}  of Proposition \ref{prop: matrix V}.
 We will prove the proposition for vectors $v_1 \etc v_l$ with such restricted norms first and derive the general statement by taking the union bound over $d_1 \etc d_l$ being dyadic integers.

We begin the proof of Proposition \ref{prop: matrix V} with constructing a special net for the set $W_{\mathbf{d}}$.
This will follow by proving that the $l$-tuples from the net approximate any point of $W_{\mathbf{d}}$ in a number of senses. After that we will prove the individual small ball probability estimate for \emph{some} $l$-tuples from the net. These will be exactly those tuples that appear as a result of approximation of points from $W_{\mathbf{d}}$.

To make the construction of the net simpler, we introduce another parameter. 
Given $\rho$ as in Proposition \ref{prop: matrix V}, we will chose $\d>0$ such that 
\begin{equation} \label{eq: delta}
   \d \le  \rho \quad \text{and }  \d^{-1} \in \N.
\end{equation}
The parameter $\d$ will be adjusted several times throughout the proof, but its value will remain independent of $n$.
\begin{lemma}[Net cardinality] \label{lem: net card}
	Let $\mathbf{d}=(d_1 \etc d_l)$ be a vector such that $d_j \in [r \sqrt{n},R]$ for all $j \in [l]$. 
	 Let $\d$ be as in \eqref{eq: delta}, and $\NN_{\mathbf{d}} \subset (\d \Z^n)^l$ be the set of all $l$-tuples of vectors $u_1 \etc u_l$ such that 
    \[
     \norm{u_j}_2 \in \left[\frac{1}{2} d_j, 4 d_j\right] \quad \text{for all }\ j \in [l]
    \]	 	
    and 
    \[
     \dist (u_j, \Z^n) \le 2 \rho \sqrt{n}.
    \]
    Then
    \[
      |\NN_{\mathbf{d}}|
      \le  \left( \frac{C \rho}{r \delta}\right)^{l n} \left( \prod_{j=1}^l \frac{d_j}{\sqrt{n}}\right)^n.
    \]
\end{lemma}

\begin{proof}
	  Let $\MM_j= \Z^n \cap 2 d_j B_2^n$.
    Taking into account that $d_j \ge r \sqrt{n}$, we use Lemma \ref{lem: integer inside ball} to conclude that
	\[
	|\MM_j|
	\le \left(2+ \frac{C d_j}{\sqrt{n}}\right)^n
	\le  \left( \frac{C'}{r}\right)^n \cdot  \left(\frac{d_j}{\sqrt{n}} \right)^n.
	\]

	Define the set $\MM$ by $\MM = \d  \, \Z^n \cap 2 \rho \sqrt{n} B_2^n$. 
	Similarly, Lemma \ref{lem: integer inside ball} yields
	\[
	|\MM| \le \left( \frac{C \rho}{\d } \right)^n.
	\]
	Set $\NN_j=\MM_j+\MM \subset \d \, \Z^n$.
	Here we used the assumption that $\d^{-1} \in \N$.
	Then by construction
	\[
	|\NN_j| 
	\le \left( \frac{C \rho }{ r \d }  \right)^n  \cdot \left( \frac{d_j}{  \sqrt{ n}} \right)^n.
	\]
	Set $\NN_{\mathbf{d}}=\prod_{j=1}^l \NN_j$.
	Multiplying the previous estimates, we obtain
	\[
	|\NN_{\mathbf{d}}| 
	\le  \left( \frac{C \rho}{r \d}\right)^{l n} \left( \prod_{j=1}^l \frac{d_j}{\sqrt{n}}\right)^n
	\]
	as required.	
\end{proof}

The next step is the central technical part of this section.
Our next task is to show that for any $(v_1 \etc v_l) \in W_{\mathbf{d}}$, there exists a sequence $(u_1 \etc u_l) \in \NN_{\mathbf{d}}$ which approximates it in various ways. As some of these approximations hold only for a randomly chosen point of $\NN_{\mathbf{d}}$, and we need all of them to hold simultaneously, we have to establish all of them at the same time.
This will be done by using random rounding as in the proof of Proposition \ref{prop: compressible}. 
The implementation of this method here is somewhat different since we have to control the least common denominator of the matrix $U$ formed by the vectors $u_1 \etc u_l$.

We will prove the following lemma.

\begin{lemma}[Approximation] \label{lem: approx}
	Let $k \le cn$.
	Let $\mathbf{d}=(d_1 \etc d_l) \in [r \sqrt{n}, R]^l$.
	Let $\d>0$ be a sufficiently small constant satisfying \eqref{eq: delta}.
	Let $B$ be an $(n-k) \times n$ matrix such that $\norm{B}_{\HS} \le 2Kn$.
	For any sequence $(v_1 \etc v_l) \in W_{\mathbf{d}} \cap \Ker(B)$, there exists a sequence $(u_1 \etc u_l) \in \NN_{\mathbf{d}}$ with the following properties 
	\begin{enumerate}
		\item \label{it: app1} $\norm{u_j-v_j}_{\infty} \le \d$ for all $j \in [l]$;
		\item \label{it: app2} Let $U$ and $V$ be $n \times l$ matrices with columns $u_1 \etc u_l$ and $v_1 \etc v_l$ respectively. Then 
		\[
		 \norm{U-V} \le C \d \sqrt{n}.
		\]
		\item \label{it: app3} the system $(u_1 \etc u_l)$ is $(1/4)$-orthogonal;
		\item \label{it: app35} $\Span(u_1 \etc u_l ) \cap S^{n-1} \subset \Incomp(\tau^2, \tau^4/2)$;
		\item \label{it: app4} $\dist(u_j, \Z^n) \le 2 \rho \sqrt{n}$ for all $j \in [n]$;
		\item \label{it: app5} Let $U$ be as in \eqref{it: app2}. Then
		\[
		 \dist(U \theta, \Z^n) > \frac{\rho}{2} \sqrt{n}
		\]
		for any $\theta \in \R^n$ satisfying 
		\[
		  \norm{\theta}_2 \le \frac{1}{20 \sqrt{l}}  \quad
		  \text{ and} \quad \norm{U \theta}_2 \ge 8 r \sqrt{n};
		\]
		\item \label{it: app6} $\norm{B u_j}_2 \le 2 K \d n$ for all $j \in [l]$.
	\end{enumerate}
\end{lemma}

The conditions \eqref{it: app3}-\eqref{it: app5} are the same as \eqref{it: 15}-\eqref{it: 4} of Proposition \ref{prop: matrix V} up to a relaxation of some parameters.

\begin{proof}
	Let $(v_1 \etc v_l) \in W_{\mathbf{d}}$. 
	Choose $(v_1' \etc v_l') \in \d \Z^n$ be such that 
	\[
	v_j \in v_j' + \d [0,1]^n \quad \text{for all } j \in [l].
	\]
	Define independent random variables $\e_{i,j}, i \in [n], j \in [l]$ by setting
	\[
	\P (\e_{i,j} = v_j'(i)-v_j(i))= 1- \frac{v_j(i)-v_j'(i)}{\d} 
	\]
	and
	\[
	  \P (\e_{i,j} = v_j'(i)-v_j(i) +\d)= \frac{v_j(i)-v_j'(i)}{\d} .
	\]
	Then $|\e_{i,j}| \le \d$ and $\E \e_{i,j}=0$. 
	Consider a random point 
	\[
	 u_j = v_j + \sum_{i=1}^n \e_{i,j} e_i \in \d^{-1} \Z^n.
	\]
	Then $\E u_j=v_j$ and $\norm{u_j-v_j}_{\infty} \le \d$ for all $j \in [l]$ as in \eqref{it: app1}.
	Let us check that  $(u_1 \etc u_l) \in \NN_{\mathbf{d}}$ for any choice of $\e_{i,j}$.
	Indeed, for any $j \in [l]$,
	\[
	\norm{u_j-v_j}_2 \le \d \sqrt{n} \quad \text{and } \left( 1- \frac{\d}{r} \right) \norm{v_j}_2 \le \norm{u_j}_2 \le \left( 1+ \frac{\d}{r} \right) \norm{v_j}_2
	\]
	as $\norm{ v_j}_2  \ge r \sqrt{n}$ for all $j \in [l]$.
	This, in particular, implies that $\norm{u_j}_2 \in \left[\frac{1}{2} d_j, 4 d_j\right]$ for all $j \in [l]$ and any values of $\e_{i,j}$.

	Let $U$ and $V$ be the $n \times l$ matrices with columns $u_1 \etc u_l$ and $v_1 \etc v_l$ respectively.
	Then the matrix $U-V$ has independent entries $\e_{i,j}, \ i \in [n], j \in [l]$ which are centered and bounded by $\d$ in the absolute value.
	By Lemma \ref{lem: operator norm},
	\[
	\P(\norm{U-V} \ge C_{\ref{lem: operator norm}} \d \sqrt{n}) \le \exp(-c_{\ref{lem: operator norm}} n),
	\]
	and so condition \eqref{it: app2} holds with probability at least $1-\exp(-c_{\ref{lem: operator norm}} n)$.

	Let us check that condition \eqref{it: app3} follows from \eqref{it: app2}.
	Let $D_U$ be the diagonal matrix $D_U= \diag (\norm{u_1}_2 \etc \norm{u_l}_2)$, and define $D_V$ in a similar way.
	If $\norm{U-V} \le C_{\ref{lem: operator norm}} \d \sqrt{n}$, then by the $\left(\frac{1}{8}\right)$-almost orthogonality of $(v_1 \etc v_l)$, we get
		\begin{align*}
		\norm{U D_U^{-1}}
		&\le \norm{U D_V^{-1}} \cdot \norm{D_V D_U^{-1}}  \\
		&\le \big[  \norm{V D_V^{-1}} + \norm{U-V} \cdot \norm{D_V^{-1}}  \big] \cdot \norm{D_V D_U^{-1}}\\
		& \le \left[ \frac{9}{8} + C_{\ref{lem: operator norm}} \d \sqrt{n}\cdot \frac{1}{r \sqrt{n}}  \right] \cdot \left( 1 - \frac{\d}{r}\right)^{-1} 
		\le \frac{5}{4}
	\end{align*}
	if $\d \le c r$ for an appropriately small constant $c>0$.
    Similarly,
	\begin{align*}
		s_l(U D_U^{-1}) 
		&\ge s_l(U D_V^{-1}) \norm{D_U D_V^{-1}}^{-1}  \\
		&\ge \big[ s_l(V D_V^{-1}) - \norm{U-V} \cdot \norm{D_V^{-1}} \big] \cdot \norm{D_U D_V^{-1}}^{-1}  \\
		&\ge \left[ \frac{7}{8}  - C_{\ref{lem: operator norm}}  \frac{\d}{r} \right] \cdot \left( 1 + \frac{\d}{r}\right)^{-1} 
		\ge \frac{3}{4}
	\end{align*}
    confirming our claim.
    The last inequality above follows again by choosing $\d <cr$ with a sufficiently small $c>0$.
    
    Let us check that condition \eqref{it: app35} follows from \eqref{it: app2} and \eqref{it: app3}.
    Indeed, let $\theta \in \R^l$ be such that $\norm{U \theta}_2=1$.
    Since $\min(\norm{u_1}_2 \etc \norm{u_l}_2 ) \ge r \sqrt{n}$, and the system $u_1 \etc u_l$ is $\left(\frac{1}{4}\right)$-almost orthogonal, we have 
    \[
      \norm{\theta}_2 \le \big(s_l(U)\big)^{-1} \norm{U \theta}_2 \le
       \frac{4}{3} \cdot \frac{1}{r \sqrt{n}}.
    \] 
    At the same time,
    \[
    \norm{V \theta}_2 \ge \norm{U \theta}_2 - \norm{U - V} \cdot \norm{\theta}_2
    \ge 1 - C_{\ref{lem: operator norm}} \d \sqrt{n} \cdot  \frac{4}{3} \cdot \frac{1}{r \sqrt{n}}
    =1- C_{\ref{lem: operator norm}} \frac{4 \d}{3 r}.
    \]
    Take any $y \in \Sparse (\tau^2 n)$. Then
    \[
     \norm{V \theta -y}_2 
     \ge \left(1-C_{\ref{lem: operator norm}}\frac{4 \d}{3 r}\right) \cdot \norm{\frac{V \theta}{\norm{V \theta}_2} - \frac{y}{\norm{V \theta}_2}}_2
     \ge \left(1-C_{\ref{lem: operator norm}}\frac{4 \d}{3 r}\right) \cdot \tau^4
    \]
    since $\frac{V \theta}{\norm{V \theta}_2} \in \Incomp(\tau^2 n, \tau^4)$.
    Therefore,
    \[
      \norm{U \theta -y}_2 
      \ge \norm{V \theta -y}_2 - \norm{U-V} \cdot \norm{\theta}_2
      \ge  \left(1-C_{\ref{lem: operator norm}}\frac{4 \d}{3 r}\right) \cdot \tau^4 - C_{\ref{lem: operator norm}} \frac{4 \d}{3 r}
      \ge \frac{1}{2} \tau^4
    \]
    where the last inequality holds if $\d$ is appropriately adjusted depending on $r$ and $\tau$.
    This proves that if $\norm{U\theta}_2=1$ then $\dist(U \theta, \Sparse(\tau^2 n)) \ge \tau^4/2$, i.e., $U \theta \in \Incomp(\tau^ n, \tau^4/2)$. Thus \eqref{it: app35} is verified.
    
    Condition \eqref{it: app4} immediately follows from \eqref{it: app1} and the triangle inequality:
    \[
     \dist(u_j, \Z^n) \le \dist(v_j, \Z^n) + \d \sqrt{n} \le 2 \rho \sqrt{n}
    \]
    since $\d \le \rho$.
    
    Condition \eqref{it: app5} follows from \eqref{it: app2} and \eqref{it: app3}. Indeed, let $\theta$ be as in \eqref{it: app5}, and assume that 
    $\norm{U-V} \le C_{\ref{lem: operator norm}} \d \sqrt{n}$.
    Since both $(v_1 \etc v_l)$ and $(u_1 \etc u_l)$ are $\left(\frac{1}{4}\right)$-almost orthogonal,  and $\norm{u_j}_2 \ge \frac{1}{2} \norm{v_j}_2$,
    \[
    \norm{V \theta}_2^2 \ge \frac{1}{4} \sum_{j=1}^l \theta_j^2 \norm{v_j}_2^2 
    \ge  \frac{1}{16} \sum_{j=1}^l \theta_j^2 \norm{u_j}_2^2
    \ge \frac{1}{64} \norm{U \theta}_2^2 \ge  r^2 n.
    \]
    As $(v_1 \etc v_l) \in W_{\mathbf{d}}$, this implies that 
    \[
    \dist(V \theta, \Z^n) > \rho \sqrt{n}.
    \]

    Therefore,
    \begin{align*}
    	\dist (U \theta, \Z^n)
    	& \ge  \dist (V \theta, \Z^n) - \norm{(U-V)^\top \theta}_2 \\
    	&>   \rho  \sqrt{n} - \norm{U-V} \cdot \norm{\theta}_2
    	\ge   \rho \sqrt{n} - C_{\ref{lem: operator norm}} \d \sqrt{n} \cdot \frac{1}{20 \sqrt{l} } \\
    	&\ge  \frac{\rho}{2}  \sqrt{n},
    \end{align*}
    where we  adjust $\d$ again if necessary.   
	Thus $(u_1 \etc u_l)$ satisfy \eqref{it: app2}--\eqref{it: app5} with probability at least $1-\exp(-c_{\ref{lem: operator norm}} n)$.

	It remains to show that we can choose $(u_1 \etc u_l)$ satisfying \eqref{it: app6} at the same time.
	For any $j \in [l]$, we have 
	\[
	\E \norm{B (u_j-v_j)}_2^2 = \E \norm{\sum_{i=1}^n \e_{i,j} B e_i}_2^2
	= \sum_{i=1}^n \E \e_{i,j}^2 \norm{B e_i}_2^2 
	\le \frac{\d^2}{4} \norm{B}_{\HS}^2 \le K^2 \d^2  n^2.
	\]
	By Chebyshev's inequality,
	\[
	\P ( \norm{B (u_j-v_j)}_2 \le 2 K \d n ) \ge \frac{1}{2}.
	\]
	In view of independence of these events for different $j$,
	\[
	\P (\forall j \in [l] \  \norm{B (u_j-v_j)}_2 \le 2 K \d n ) \ge 2^{-l}.
	\] 
	As 
	\[
	  1-\exp(-c_{\ref{lem: operator norm}} n)+2^{-l}>1, 
	\]
	there is a realization $(u_1 \etc u_l) \in \NN_{\mathbf{d}}$ satisfying \eqref{it: app2}--\eqref{it: app5} for which
	\[
	  \norm{B u_j}_2= \norm{B (v_j-w_j)}_2  \le 2 K \d n
	\]
	 holds for all $j \in [l]$ simultaneously.
	This finishes the proof of the lemma.
\end{proof}

 Fix the value of $\d$ satisfying \eqref{eq: delta} such that Lemma \ref{lem: approx} holds for the rest of the proof.

  We will now use the small ball probability estimate of Theorem \ref{thm: SBP LCD} to show that the event $W_{\mathbf{d}} \cap \Ker(B) \neq \varnothing$ is unlikely.
  \begin{lemma} \label{lem: sbp on the net}
  	Let  $\mathbf{d}=(d_1 \etc d_l) \in [r \sqrt{n}, R]^l$ where $R,r$ are defined above.
  	Let $k \le \frac{\d}{20} \sqrt{n}$ and $\frac{k}{10} \le l \le k$. Then
  	\[
  	  \P \left(  W_{\mathbf{d}} \cap \Ker(B) \neq \varnothing \right)
  	  \le \exp (- 2l n).
  	\]
  \end{lemma}

\begin{proof}

	Let $\NN_{\mathbf{d}}$ be the net constructed in Lemma \ref{lem: net card}.
	Let $\tilde{\NN_{\mathbf{d}} }$ be the set of all $(u_1 \etc u_l) \in \NN_{\mathbf{d}}$ which satisfy conditions \eqref{it: app3} -- \eqref{it: app5} of Lemma \ref{lem: approx}.
	Consider an $l$-tuple $u_1 \etc u_l \in \tilde{\NN_{\mathbf{d}} } $.
	Let $U$ be the $n \times l$ matrix with columns $u_1 \etc u_l$.

	To apply the Levy concentration estimate of Theorem \ref{thm: SBP LCD}, we have to bound the LCD of $U^\top$ from below.
	Let us show that
	\begin{equation}  \label{eq: lower LCD}
	D_{L,\a}(U^\top) \ge \frac{1}{20 \sqrt{l}}.
	\end{equation}
	Take $\theta \in \R^l$ such that $\norm{\theta}_2 \le \frac{1}{20 \sqrt{l}}$.
	Assume first that 
	\[
	  \norm{U \theta}_2 \le 8 r \sqrt{n} = \sqrt{\tau^2 n}/2.
	\]
	Recall that $L$ and $\a$ are defined as in \eqref{eq: def L}.
	Since the columns of $U$ satisfy \eqref{it: app35} of Lemma \ref{lem: approx}, applying Lemma \ref{lem: LCD compressible} yields
	\[
    	\dist(U \theta, \Z^n) \ge L \sqrt{\log_+ \frac{\a \norm{U \theta}_2}{L} }.
	\]
	
	Assume now that $\norm{U \theta}_2 > 8 r \sqrt{n}$.	
	By condition \eqref{it: app1} of Lemma \ref{lem: approx},
	\[
	 \norm{U}_{\HS} \le \sqrt{l} \max_{j \in [l]} \norm{u_j}_2 \le \sqrt{l} R.
	\]
	Hence,
   \[
   	L \sqrt{\log_+ \frac{\a \norm{U \theta}_2}{L}} 
	\le L  \sqrt{\log_+\frac{\norm{U}_{\HS}}{L}}
	\le L  \sqrt{\log_+ R}
	\le L \sqrt{\frac{\rho^2}{4} \cdot \frac{n}{L^2}} \le \frac{\rho}{2} \sqrt{n}.
	\]
	By condition \eqref{it: app5} of Lemma \ref{lem: approx},
	\[
	\dist(U \theta, \Z^n) > \frac{\rho}{2} \sqrt{n}
	\]
	whenever $\theta \in \R^n$ satisfies
	\[
	\norm{\theta}_2 \le \frac{1}{20 \sqrt{l}}  \quad
	\text{ and} \quad \norm{U \theta}_2 \ge 8 r \sqrt{n}.
	\]
	Combining these two cases, we see that any vector $\theta \in \R^l$ with 
	$	\norm{\theta}_2 \le \frac{1}{20 \sqrt{l}}$ satisfies
	\[
    	\dist(U \theta, \Z^n) \ge L \sqrt{\log_+ \frac{\a \norm{U \theta}_2}{L} }
	\]
	which proves \eqref{eq: lower LCD}.

	Using condition \eqref{it: app3} of Lemma \ref{lem: approx} and Lemma \ref{lem: almost ort}, we infer
	\[
	\det(U^\top U)^{1/2} \ge 4^{-l} \prod_{j=1}^l \norm{u_j}_2 \ge 8^{-l} \prod_{j=1}^l d_j.
	\]

	Let $i \in [n]$. Recall that $\Row_i(B) \in \R^n$ is a vector with i.i.d. random coordinates satisfying \eqref{eq: xi} and that $l \le k \le \frac{\d}{20 \sqrt{n}}$. 
	Combining this with \eqref{eq: SBP LCD} used with 
	\[
	t \ge \d \sqrt{n} \ge 20 l \ge \frac{\sqrt{l}}{D_{L,\a}(U^\top)}
	\]
	and recalling that $L=O(\sqrt{k})$ by \eqref{eq: def L} and $l \ge k/10$, we obtain
	\[
	\P \left(  \norm{U^\top (\Row_i(B))^\top}_2 \le t \sqrt{l }\right)  
	\le \frac{(CL/\sqrt{l})^l}{\det(U^\top U)^{1/2}} \left(t + \frac{\sqrt{l}}{D_{L,\a}(U^\top)} \right)^l
	\le  \frac{C^l}{\prod_{j=1}^l d_j} t^l.
	\]

	Denote
	\[
	Y_i = \frac{1}{l}\norm{U^\top (\Row_i(B))^\top}_2^2, \qquad
	M=  \frac{C^2}{\left( \prod_{j=1}^l d_j \right)^{2/l}}.
	\]
	Then we can rewrite the last inequality as
	\[
	\P(Y_i \le s) \le (Ms)^{l/2} \quad \text{for } s \ge s_0= \d^2 n.
	\]
	In view of Lemma \ref{lem: tenzorization} applied with $m=l/2$ and $t= 4K^2 s_0$ with $K$ from \eqref{eq: xi}, this yields
	\begin{align*}
		\P \left( \norm{Bu_j}_2 \le 2K \d n \text{ for all } j \in [l] \right)
		&\le  \P \left( \sum_{j=1}^l  \norm{Bu_j}_2^2 \le 4K^2 \d^2 l n^2 \right)  \\
		&= \P \left( \sum_{i=1}^{n-k}   \norm{U^\top (\Row_i(B))^\top}_2 ^2 \le  4 K^2 \d^2 l n^2 \right)  \\
		&= \P \left( \sum_{i=1}^{n-k} Y_i \le n \cdot 4K^2 \d^2  n  \right) 
		\le (C' M \d^2  n)^{(n-k) l/2}\\
		&=  (C'' \d)^{l(n-k)} \cdot \left( \prod_{j=1}^l \frac{\sqrt{ n}}{d_j}\right)^{n-k}.
	\end{align*}

	Since $\tilde{\NN_{\mathbf{d}} } \subset \NN_{\mathbf{d}}$,
	a combination of the  small ball probability estimate above and Lemma \ref{lem: net card}  gives
	\begin{align*}
		&\P \left( \exists (u_1 \etc u_l) \in \tilde{\NN_{\mathbf{d}} }: \ \norm{Bu_j}_2 \le  \d n, \ j \in [l] \right) \\
		&\le |\NN_{\mathbf{d}}| \cdot   (C'' \d)^{l (n-k)} \cdot \left( \prod_{j=1}^l \frac{\sqrt{n}}{d_j}\right)^{n-k} \\
		&\le  \left( \frac{C \rho}{r \d}\right)^{l n} \left( \prod_{j=1}^l \frac{d_j}{\sqrt{n}}\right)^n
		\cdot (C'' \d)^{l (n-k)} \cdot  \left( \prod_{j=1}^l \frac{\sqrt{n}}{d_j}\right)^{n-k} \\
		&=   \left( \frac{C' \rho}{r }\right)^{l n} \cdot \d^{-lk}  \left( \prod_{j=1}^l \frac{d_j}{\sqrt{n}}\right)^k.
	\end{align*}
	Recall that by  \eqref{eq: def r,R},
	\[
	 d_j \le R= \exp \left(\frac{\rho^2 n}{4 L^2}\right) \le  \exp \left(\frac{C \rho^2 n}{k}\right)   \quad \text{for all } j \in [l],
	\]
	where the last inequality follows from \eqref{eq: def L}.
	Therefore,
	\begin{align*}
	\P \left( \norm{Bu_j}_2 \le 2K \d n \  \text{ for all } j \in [l] \right)
	&\le   \left( \frac{C' \rho}{r }\right)^{l n} \cdot  \left(\frac{R}{\d \sqrt{n}}\right)^{lk}
	\le   \left( \frac{\tilde{C} \rho}{r } \exp(C \rho^2)\right)^{l n}  \\
	&\le \exp (-2ln)
\end{align*}	
	if $\rho <c r$ for a sufficiently small constant $c>0$.

	Notice that
	\begin{align*}
		\P \left(  W_{\mathbf{d}} \cap \Ker(B) \neq \varnothing \right)
		&\le \P \left(  W_{\mathbf{d}} \cap \Ker(B) \neq \varnothing \text{ and } \norm{B}_{\HS} \le 2 K n 
		\right) \\
		&+ \P \left( \norm{B}_{\HS} \ge 2 K n \right).
	\end{align*}
	In view of Lemma \ref{lem: HS norm}, the second term is smaller than $\exp(-cn^2)$ which means that we have to concentrate on the first one.
	
	Assume that the events $ W_{\mathbf{d}} \cap \Ker(B) \neq \varnothing \text{ and } \norm{B}_{\HS} \le  2 K n$ occur, and pick an $l$-tuple $(v_1 \etc v_l) \in  W_{\mathbf{d}} \cap \Ker(B)$.
	Choose an approximating $l$-tuple $(u_1 \etc u_l) \in \NN_{\mathbf{d}}$ as in Lemma \ref{lem: approx}. Then $(u_1 \etc u_l) \in \tilde{\NN_{\mathbf{d}}}$ and $\norm{B u_j}_2 \le 2 K \d n$ per condition \eqref{it: app6} of this lemma.
	The argument above shows that the probability of the event that such a tuple $(u_1 \etc u_l) \in \tilde{\NN_{\mathbf{d}}}$ exists is at most $\exp(-2ln)$.
	The lemma is proved.	
\end{proof}

Proposition \ref{prop: matrix V} follows from Lemma \ref{lem: sbp on the net} by taking the union bound over dyadic values of the coordinates of $\mathbf{d}$.
\begin{proof}[Proof of Proposition \ref{prop: matrix V}]
	Let $\EE_{\mathbf{d}}$ be the event that $ W_{\mathbf{d}} \cap \Ker(B) \neq \varnothing$. 
	 Then 
	\[
	\EE_l= \bigcup \EE_{\mathbf{d}},
	\]
	where the union is taken over all vectors $\mathbf{d}$ with dyadic coordinates: $d_j=2^{s_j}, \ s_j \in \N$ such that $2^{s_j} \in [r \sqrt{n}, R]$. Since there are at most 
	\[
	 \left[ \log \left( \frac{2R}{r \sqrt{n}}\right) \right]^l
	\le  \left( \frac{C \rho^2 n}{L^2}\right)^l
	\]
	terms in the union, Lemma \ref{lem: sbp on the net} yields
	\[
		\P(\EE_{\mathbf{d}}) \le \left( \frac{C \rho^2 n}{L^2}\right)^l  \exp \left( -2 l n \right)
		\le \exp \left( - l n \right)
	\]
	where we took into account that $L >1$. This finishes the proof of the proposition.
\end{proof}

\section{Rank of a random matrix} \label{sec: Rank}

 We will complete the proof of Theorem \ref{thm: rank} using the probability estimates of Propositions \ref{prop: compressible} and \ref{prop: matrix V}.
 These propositions show that the linear subspace spanned by the first $n-k$ columns of the matrix $A$ is unlikely to contain a large almost orthogonal system of vectors with a small or moderate least common denominator. 
 Applying Lemma \ref{lem: min config}, we will show that with high probability, this subspace contains a further subspace of a dimension proportional to $k$ which has no vectors with a subexponential least common denominator. 
 The next lemma shows that in such a typical situation, it is unlikely that the rank of the matrix $A$ is $n-k$ or smaller.

\begin{lemma}  \label{lem: large LCD}
	Let $A$ be an $n \times n$ random matrix whose entries are independent copies of a  random variable $\xi$ satisfying \eqref{eq: xi}.
	For $k< \sqrt{n}$ define
	\[
	\Omega_k=\Omega_k(\Col_1(A) \etc \Col_{n-k}(A))
	\]
	as the event that there exists a linear subspace $E \subset \big(\Span(\Col_1(A) \etc \Col_{n-k}(A)) \big)^\perp$ such that $\dim(E) \ge k/2$ and 
	\[
	D_{L, \a }(E) \ge \exp \left(  C \frac{n}{k} \right).
	\]
	Then
	\begin{align*}
	&\P \big( \Col_j(A) \in \Span (\Col_i(A), \ i \in [n-k]) \text{ for } j=n-k+1 \etc n \text{ and } \Omega_k \big) \\
	&\le \exp(-c' n k).
	\end{align*}
\end{lemma}

\begin{proof}
	Assume that $\Omega_k$ occurs.
	The subspace $E$ can be selected in a measurable way with respect to the sigma-algebra generated by $\Col_1(A) \etc \Col_{n-k}(A)$.
	Therefore, conditioning on  $\Col_1(A) \etc \Col_{n-k}(A)$ fixes this subspace.
	Denote the orthogonal projection on the space $E$ by $P_E$.
	 Since $E$ is independent of $\Col_{n-k+1}(A) \etc \Col_n(A)$, and these columns are mutually independent as well, it is enough to prove that 
	\begin{align*}
	  &\P (\Col_j(A) \in \Span (\Col_i(A), \ i \in [n-k]) \text{ for } j=n-k+1 \etc n \mid E   )  \\
		& \le \P (\Col_j(A) \in E^\perp  \text{ for } j =n-k+1 \etc n \mid E   ) \\
		& = \big(\P (P_E \Col_n(A) =0 \mid E   ) \big)^k \\
		&\le \exp(-c n k),
	\end{align*}
	or
	\[
	\P (P_E \Col_n(A) = 0  \mid E   ) 
	\le \exp(-c n ).
	\]
	Using Corollary \ref{cor: SBP proj} with $m=k/2$ and $t=0$, we obtain
	\[
	\P (P_E \Col_n(A) = 0 \mid E  )
	\le C^m \left( \sqrt{m}  \exp \left( - C \frac{n}{k} \right) \right)^m
	\le \exp(-c n )
	\]
	as required.
\end{proof}

 With all ingredients in place, we are now ready to prove the main theorem. 
\begin{proof}[Proof of Theorem \ref{thm: rank}]
	Recall that it is enough to prove Theorem \ref{thm: rank} under the condition that the entries of $A$ are i.i.d. copies of a random variable satisfying \eqref{eq: xi}.
	 
	Assume that $\rank(A) \le n-k$.
	Then there exists a set $J \subset [n], \ |J|=n-k$ such that
	$\Col_j(A) \in \Span(\Col_i(A), \ i \in J)$ for all $j \in [n] \setminus J$.
	Since the number of such sets is
	\[
	 \binom{n}{k} \le \exp \left(k \log \left(\frac{e n}{k} \right) \right)
	 \ll \exp(-c k n), 
	\]
	 it is enough to show that
	 \[
	  \P \left(\Col_j(A) \in \Span(\Col_i(A), \ i \in J) \text{ for all } J \in [n] \setminus J \right)
	  \le \exp( - c kn)
	 \]
	 for a single set $J$.
	 As the probability above is the same for all such sets $J$, without loss of generality assume that $J=[n-k]$.

	Consider the $(n-k) \times n$ matrix $B$ with rows $\Row_j(B)=(\Col_j(A))^\top$ for $j \in [n-k]$.
	Let $E_0= \Ker (B)$ and denote by $P_{E_0}$ the orthogonal projection onto $E_0$.
	Then the condition $\Col_j(A) \in \Span(\Col_i(A), \ i \in [n-k])$ reads $P_{E_0} \Col_j(A)=0$.
	
	Let $\tau$ be the constant appearing in Proposition \ref{prop: compressible}, and denote
	\[
	  W_0= \Comp (\tau^2n, \tau^4).
	\]
	Set $l=k/4$.
	Lemma \ref{lem: min config} asserts that at least one of the events described in \eqref{it: ort system} and \eqref{it: subspace} of this lemma occurs. Denote these events $\EE_{\ref{lem: min config}}^{\eqref{it: ort system}}$ and $\EE_{\ref{lem: min config}}^{\eqref{it: subspace}}$ respectively.
	In view of Proposition \ref{prop: compressible},
	\[
	  \P (\EE_{\ref{lem: min config}}^{\eqref{it: ort system}})
	  \le \exp \left( -\frac{k}{4} n\right).
	\]
	Here we used only condition \eqref{it: 13} in Lemma \ref{lem: min config} ignoring condition \eqref{it: 14}.

	Assume now that $\EE_{\ref{lem: min config}}^{\eqref{it: subspace}}$ occurs and consider the subspace $F \subset E_0, \ \dim (F)= \frac{3}{4}k$ such that $F \cap  \Comp (\tau^2n, \tau^4) = \varnothing$.
	Let $\rho$ be the constant appearing in Proposition \ref{prop: matrix V}, and let $L$ be as in \eqref{eq: def L}.
	Set 
	\[
	 W_1= \left\{ v \in F: \ \frac{\tau}{8} \sqrt{n} \le \norm{v}_2 \le \exp \left( \frac{\rho^2 n}{4L^2} \right) \text{ and } \dist (v, \Z^n) \le \rho \sqrt{n} \right\}
	\]
	Applying Lemma \ref{lem: min config} to  $W_1$ and $l= \frac{k}{4}$, we again conclude that one of the following events occurs:
	\begin{enumerate}
		\item there exist vectors $v_1 \etc v_{k/4} \in F \cap W_1$ such that 
		\begin{enumerate}
			\item the $(k/4)$-tuple $(v_1 \etc v_{k/4})$ is $\left(\frac{1}{8}\right)$-almost orthogonal and
			\item for any $\theta \in \R^{k/4}$ with 
			\[
			  \norm{\theta}_2 \le \frac{1}{20 \sqrt{k/4}},
			\]
			$\sum_{i=1}^{k/4} \theta_i v_i \notin W_1$
		\end{enumerate}
    \vskip 0.1in
	or 
    \vskip 0.1in
	   \item there is a subspace $\tilde{F} \subset F$ with $\dim (\tilde{F}) = \frac{k}{2}$ such that $\tilde{F} \cap W_1 = \varnothing$.
	\end{enumerate}
    Denote these events $\VV_{\ref{lem: min config}}^{\eqref{it: ort system}}$ and $\VV_{\ref{lem: min config}}^{\eqref{it: subspace}}$ respectively.
    In view of Proposition \ref{prop: matrix V}, 
    \[
      \P (\VV_{\ref{lem: min config}}^{\eqref{it: ort system}})
      \le \exp \left( -\frac{k}{4} n\right).
    \]
    
    Assume now that the event $\VV_{\ref{lem: min config}}^{\eqref{it: subspace}}$ occurs.
    We claim that in this case,
    \[
      D_{L, \a}(\tilde{F}) \ge R: = \exp \left( \frac{\rho^2 n}{4L^2} \right).
    \]
    The proof is similar to the argument used in the proof of Lemma \ref{lem: sbp on the net}.
    Let $S: \R^{k/2} \to \R^n$ be an isometric embedding such that $S \R^{k/2} = \tilde{F}$. 
    Then $D_{L,\a} (\tilde{F})= D_{L,\a} (S^\top)$.
    Let $\theta \in \R^{k/2}$ be a vector such that 
    \[
      \dist(S \theta, \Z^n) < L \sqrt{\log_+ \frac{\a \norm{\theta}_2}{L}}.
    \]
    Since
    \[
    S \R^{k/2} \cap S^{n-1} \subset F \cap S^{n-1} \subset \Incomp(\tau^2 n, \tau^4),
    \]
    Lemma \ref{lem: LCD compressible} applied with $U=S$ and $s=\tau^2$ yields
    \[
     \norm{\theta}_2 \ge \tau \sqrt{n}/2.
    \]
    On the other hand, if $\norm{\theta}_2 \le R$, then 
    \[
       L \sqrt{\log_+ \frac{\a \norm{\theta}_2}{L}}
       \le \rho \sqrt{n},
    \]
    and therefore $\dist(S \theta, \Z^n) < \rho \sqrt{n}$.
    Since $\tilde{F}=S \R^{k/2} \cap W_1 = \varnothing$, this implies that 
    \[
      \norm{\theta}_2 = \norm{S \theta}_2 > R =  \exp \left( \frac{\rho^2 n}{4L^2} \right),
    \] 
    thus proving our claim and checking the assumption of Lemma \ref{lem: large LCD}.

    Finally, 
	 \begin{align*}
       &\P \left(\Col_j(A) \in \Span(\Col_i(A), \ i \in [n-k]) \text{ for } j=n-k+1 \etc n \right) \\
       &\le 2 \exp( - \frac{k}{4} n) \\
        &+
       \P( \Col_j(A) \in \Span(\Col_i(A), \ i \in [n-k])  \text{ for } j=n-k+1 \etc n  \text{ and }\VV_{\ref{lem: min config}}^{\eqref{it: subspace}}).
     \end{align*}
  Lemma \ref{lem: large LCD} shows that the last probability does not exceed $\exp(-c' (k/2) n)$. The proof is complete.
\end{proof}

 After the theorem is proved, we can derive an application to the question of Feige and Lellouche.
 \begin{lemma}  \label{lem: submatrices}
 	Let $q \in (0,1)$, and $m,n \in \N$ be numbers such that
 	\[
 	 m \le n \le \exp \left(C'_q \sqrt{m} \right).
 	\] 	
 	Let $A$ be an $m \times n$ matrix with independent Bernoulli$(q)$ entries.
 	Then with probability at least $1-\exp(- m \log n)$ all $m \times m$ submatrices of $A$ have rank greater than $m - C_q \log n$.

 	Furthermore, if $n \ge m^2$, then with probability at least $1-\exp(-cm)$, there exists an $m \times m$ submatrix $A|_S$ of $A$ with $|S|=m$ such that
 	\[
 	 \rank(A|_S) \le m - c_q \log n.
 	\]
 	The constants $C_q>c_q>0$  above can depend on $q$.
 \end{lemma}

\begin{proof}
	 The entries of $A$ are i.i.d. subgaussian random variables,  so Theorem \ref{thm: rank} applies to an $m \times m$ submatrix of $A$ as long as $k \le c \sqrt{m}$.
	 In view of the assumption of the lemma, the last inequality holds if we take $k=C_q \log n$.
	 Combining  Theorem \ref{thm: rank} with the union bound, we obtain
	 \begin{align*}
	  &	\P \left( \exists S \subset [n]: \ |S|=m \text{ and } \rank(A|_S) \le n- C_q \log n \right) \\
	  & \le \binom{n}{m} \exp (-c' m \cdot C_q \log n)
	  \le \exp \left(m \log \left(\frac{e n}{m}\right) - c' m \cdot C_q \log n \right) \\
	  & \le  \exp (- m   \log n)
\end{align*}
  if $C_q$ is chosen sufficiently large.
  
  To prove the second part of the lemma, take $k<m$ and define a random subset $J \subset [n]$ by
  \[
    J=\{j \in [n]: \ a_{1,j} = \cdots =a_{k,j}=1\}.
  \]
  Then for any $j \in [n]$,
  \[
   \P(j \in J)=q^k,
  \]
  and these events are independent for different $j \in [n]$.
  Take $k=c_q \log n$ and choose  $c_q$ so that $n q^k \ge 10 m$.
  Using Chernoff's inequality, we obtain
  \[
    \P(|J| \ge m)=1-\exp(-cm).
  \]
  On the other hand, $\rank (A|_J)\le n-k$ since this matrix contains $k$ identical rows.
  The lemma is proved.
\end{proof}

\subsection*{Acknowledgments}
This work was performed when the author held  an Erna and Jakob Michael Visiting
Professorship at the Department of Mathematics at the Weizmann Institute of Science. The author thanks the Weizmann Institute for its hospitality. He is especially grateful to Ofer Zeitouni for bringing this problem to his attention and numerous helpful discussions.



\begin{thebibliography}{S 99}
	
	\bibitem{BR invertibility} A. Basak, M. Rudelson,  \emph{Invertibility of sparse non-Hermitian matrices,} Adv. Math. 310 (2017), 426--483.
	
	\bibitem{CJMS singular} M. Campos, M. Jenssen, M. Michelen, J. Sahasrabudhe, \emph{The least singular value of a random symmetric matrix,} arXiv:2203.06141.
	
	\bibitem{FL} U. Feige, A. Lellouche, \emph{Quantitative Group Testing and the rank of random matrices}, arXiv:2006.09074.
	
	
	
	\bibitem{JSS rank} V. Jain, A. Sah, M. Sawney, \emph{Rank deficiency of random matrices,} Electron. Commun. Probab. 27 (2022), Paper No. 14, 9 pp.
	
	\bibitem{JSS part II} V. Jain, A. Sah, M. Sawney, \emph{Singularity of discrete random matrices,} 	arXiv:2010.06554.
	
	\bibitem{KKS} J. Kahn, J. Koml\'os, E. Szemer\'edi,
	{\em On the probability that a random $\pm 1$-matrix is singular},
	J. Amer. Math. Soc. 8 (1995), no. 1, 223--240.
	
	\bibitem{Komlos} J. Koml\'os,
	{\em On the determinant of $(0,\,1)$ matrices},
	Studia Sci. Math. Hungar. 2 (1967), 7--21.
	
	\bibitem{Liv} G. V. Livshyts \emph{The smallest singular value of heavy-tailed not necessarily i.i.d. random matrices via random rounding}, J. Anal. Math. 145 (2021), no. 1, 257--306.
	
	\bibitem{LTV} G. V. Livshyts, K. Tikhomirov, R. Vershynin, \emph{The smallest singular value of inhomogeneous square random matrices,} Ann. Probab. 49 (2021), no. 3, 1286--1309. 
	
	\bibitem{Nguyen} H. Nguyen, \emph{Random matrices: overcrowding estimates for the spectrum,}
	J. Funct. Anal. 275 (2018), no. 8, 2197--2224.
	
	\bibitem{RT} E. Rebrova, K. Tikhomirov, \emph{Coverings of random ellipsoids, and invertibility of matrices with i.i.d. heavy-tailed entries}, Israel J. Math. 227 (2018), no. 2, 507--544.
	
	\bibitem{R square} M. Rudelson, \emph{Invertibility of random matrices: norm of the inverse,} Ann. of Math. (2) 168 (2008), no. 2, 575--600. 
	
	\bibitem{R invertibility survey} M. Rudelson \emph{Recent developments in non-asymptotic theory of random matrices}, Modern aspects of random matrix theory, 83--120, Proc. Sympos. Appl. Math., 72, Amer. Math. Soc., Providence, RI, 2014.
	
	\bibitem{RV invertibility} M. Rudelson, R. Vershynin,
	{\em The Littlewood-Offord Problem and invertibility of random matrices},
	Adv. Math.  218  (2008),  no. 2, 600--633.
	
	\bibitem{RV Hanson-Wright} M.  Rudelson, R. Vershynin,
	\emph{Hanson-Wright inequality and sub-Gaussian concentration,} Electron. Commun. Probab. 18 (2013), no. 82, 9 pp. 
	
	\bibitem{RV no-gaps} M. Rudelson, R. Vershynin, \emph{No-gaps delocalization for general random matrices,} Geom. Funct. Anal. 26 (2016), no. 6, 1716--1776.
	
	
	\bibitem{Tikh} K. Tikhomirov, \emph{Singularity of random Bernoulli matrices,} Ann. of Math. (2) 191 (2020), no. 2, 593--634.
	
	\bibitem{V symmetric} R. Vershyinin, \emph{Invertibility of symmetric random matrices,} Random Structures Algorithms 44 (2014), no. 2, 135--182. 
	
	\bibitem{V HDP} R. Versynin, High-dimensional probability. An introduction with applications in data science. With a foreword by Sara van de Geer. Cambridge Series in Statistical and Probabilistic Mathematics, 47. Cambridge University Press, Cambridge, 2018. 
	
\end{thebibliography}



\end{document}